\documentclass[leqno,10pt]{amsart}
\usepackage[left=1.in, right=1.in, top=1.0in, bottom=.9in, headsep=.2in, footskip=0.35in]{geometry}
\input{Setup/document_style}
\input{Setup/document_macros}
\input{Setup/special_commands}

\usepackage[most]{tcolorbox}

\usepackage[backend=biber,style=alphabetic,sorting=nyt]{biblatex} 
\bibliography{bibliography.bib} 

\title{On the Morel Structure Conjecture}

\author{Giacomo Bertizzolo}

\date{\today}

\begin{document}


\begin{abstract} 
     We prove that Witt K-theory is effective over any perfect field of characteristic \(2\) and consequently finish the proof of the Morel structure conjecture initiated by Bachmann.
\end{abstract}

\maketitle

\setcounter{tocdepth}{1}
\tableofcontents

\begin{quote}
    \begin{center}
        ``I'm Hudson, sir. He's Hicks.'' - Aliens (1986).
    \end{center}
\end{quote}
\section{Introduction}
The classical Postnikov tower decomposes a spectrum into layers controlled by its homotopy groups. 
By Bott-periodicity,
the spectrum representing topological real K-theory \(\mathrm{KO}\) is eight-periodic with homotopy groups given by \(\bfZ\), \(\bfZ/2\), \(\bfZ/2\), \(0\), \(\bfZ\), \(0\), \(0\), and \(0\).\\
In motivic homotopy theory, the presence of two possible candidates for a circle, the \emph{Tate} circle \(\bfG_m\) and the simplicial circle \(S^1\), gives rise to at least two distinct Postnikov-type filtrations called  the (effective) slice filtration \cite{Voevodsky_OpenProblems} (\(\bfG_m\)-Postnikov type filtration) and the generalized (or very effective) slice filtration \cite{Rondigs_Ostvaer_SlicesHermitianK} (a \(\bfP^1\)-Postnikov type filtration).
We are interested in studying the motivic analogue of \(\mathrm{KO}\); this motivic spectrum is called the Hermitian\footnote{The underlying theory is called Grothendieck-Witt theory.} K-theory motivic spectrum \(\bfK\bfQ\).
It was first constructed in \cite{Hornbostel_HermitianK} over fields of characteristic not equal to \(2\), and in \cite{calmes2025motivicspectrumrepresentinghermitian} over any base scheme without assuming that \(2\) is invertible.
Motivic Bott-periodicity holds for \(\bfK \bfQ\) \cite[Prop. 1.0.4] {calmes2025motivicspectrumrepresentinghermitian}, which is likewise four-periodic\footnote{This four-periodicity should be interpreted as an eight-periodicity in topology, since as we previously mentioned \(\bfP^1\) is an analogue of \(S^2\)}, i.e.,
\(\bfK\bfQ \simeq \bfK\bfQ \otimes (\bfP^1)^{\otimes 4}\).
However, the slice filtration for \(\bfK\bfQ\) does not closely reflect the topological behavior of \(\mathrm{KO}\), in the sense that effective slices are not simply given by motivic cohomology \(\HZ\) (which can be defined as the zeroth slice of the sphere spectrum, the tensor-unit in \(\SH(k)\).) and \(\HZ/2\) up to twist.
The effective slices of \(\bfK\bfQ\) have been computed by R{\"o}ndigs-{\O}stav{\ae}r \cite[Thm. 4.18]{Rondigs_Ostvaer_SlicesHermitianK} and by Kolderup-R{\"o}ndigs-{\O}stav{\ae}r \cite[Cor. 4.4]{kolderup2025hermitianktheorymilnorwittmotivic}. 
Rather than exhibiting four-periodicity, the slices have a two-periodic description, that, however, involves both twisting and shifting of some components, and the addition of extra components every period.
Better behavior is expressed by the generalized slices, as proved by Bachmann \cite{Bachmann_GenSlicesKO} over a field of characteristic not equal to \(2\), and by Kolderup-R{\"o}ndigs-{\O}stav{\ae}r \cite{kolderup2025hermitianktheorymilnorwittmotivic} over essentially smooth schemes over a Dedekind scheme.
The generalized effective slices of \(\mathbf K\mathbf Q\) are four-periodic up to twists.
For \(i \equiv 1,2,3 \pmod 4\), the \(\tilde{s}_i(\bfK\bfQ)\) are simply given by some ordinary zero-slices, reflecting the vanishing of \(\pi_i KO\) for \(i=3,5,7\); furthermore, these zero-slices can be described in terms of motivic cohomology.
Instead, \(\tilde{s}_0(\bfK\bfQ)\) is an extension of two objects, mirroring the fact that both \(\pi_0 KO\) and \(\pi_1 KO\) are nonzero.\\
One might have expected that \(\tilde{s}_0(\bfK\bfQ)\) would be described as an extension with terms given by motivic cohomology, but a more complicated structure appears.
The slices fit into fiber sequences (\cite{Bachmann_GenSlicesKO}, \cite{kolderup2025hermitianktheorymilnorwittmotivic}) involving two ``modified versions'' of motivic cohomology, \(\HZtilde\) and \(\HWZ\),
\begin{equation*}
    \HZ/2[1] \to \tilde{s}_0(\bfK\bfQ) \to \HZtilde 
    \quad \text{and} \quad 
    \HWZ \otimes \bfG_m \to \tilde{s}_0(\bfK\bfQ) \to \HZ.
\end{equation*} 
These spectra are called \emph{Milnor-Witt} and \emph{Witt} motivic cohomology, respectively.
They can be defined as the effective cover of Milnor-Witt \(\rmH \bfK^{MW}\) and Witt K-theory \(\rmH \bfK^{W}\), respectively.
These definition should not appear strange, as Motivic cohomology is equivalent to the effective cover of Milnor K-theory \(\rmH \bfK^{M}\). \\
To better understand these variants of motivic cohomology, one should study them simultaneously.
The Morel Structure conjecture offers a description of the relations between these motivic spectra.
In \cite{Bachmann_GenSlicesKO}, Bachmann solved the conjecture in the characteristic not equal to \(2\) case, and this paper completes the proof of the conjecture in characteristic \(2\).
\begin{theorem}[(Morel Structure Conjecture, First Part)]\label{thm:Morel-str-conj}
    Let \(k\) be any perfect field (without restriction on the characteristic of \(k\)). Consider the square in \(\SH(k)\)
    \begin{equation}\label{eq:str-conj}
        \begin{tikzcd}[cramped]
        	{\HZtilde \simeq f^0\left(\rmH\K^{MW}\right)} && {\rmH\K^{MW}} \\
        	\\
        	{\HZ \simeq f^0\left(\rmH\K^{M}\right)} && {\rmH\K^{M}}
        	\arrow[from=1-1, to=1-3]
        	\arrow[from=1-1, to=3-1]
        	\arrow[from=1-3, to=3-3]
        	\arrow[from=3-1, to=3-3]
        \end{tikzcd}
    \end{equation}
    where:
    \begin{itemize}
        \item The right vertical map is given by \(\rmH\K^{MW} \to \rmH\K^{MW}/\eta \simeq\rmH\K^{M}\).
        \item The left vertical map is given by applying the effective cover \(f^0(\bullet)\) to the right vertical map.
        \item The horizontal maps are the counit of the adjunction 
            \begin{equation*}
                \SHeff(k) \stackrel[r^0]{i^0}{\longrightleftarrows} \SH(k).
            \end{equation*}
    \end{itemize}
    The square is a cartesian square in \(\SH(k)\).
\end{theorem}
\begin{theorem}[(Morel Structure Conjecture, Second Part)]\label{thm:Morel-str-conj-2}
    Let \(k\) be any perfect field (without restriction on the characteristic of \(k\)). Consider the square in \(\SH(k)\)
    \begin{equation}\label{eq:str-conj-2}
        \begin{tikzcd}[cramped,sep=scriptsize]
    	{\HWZ \simeq f^0(\rmH\K^{W})} &&& {\rmH\K^{W}} \\
    	\\
    	\\
    	{\HZ/2 \simeq f^0(\rmH\K^{M}/2)} &&& {\rmH\K^M/2}
    	\arrow[from=1-1, to=1-4]
    	\arrow[from=1-1, to=4-1]
    	\arrow[from=1-4, to=4-4]
    	\arrow[from=4-1, to=4-4]
    \end{tikzcd}
    \end{equation}
    where:
    \begin{itemize}
        \item The right vertical map is given by \(\rmH\K^{W} \to \rmH\K^{W}/\eta \simeq\rmH\K^{M}/2\).
        \item The left vertical map is given by applying \(f^0(\bullet)\) to the right vertical map.
        \item the horizontal maps are the counit of the adjunction 
           \begin{equation*}
                \SHeff(k) \stackrel[r^0]{i^0}{\longrightleftarrows} \SH(k).
            \end{equation*}
    \end{itemize}
    The square is a cartesian square in \(\SH(k)\).
\end{theorem}

In the case of characteristic \(2\), Theorem \ref{thm:Morel-str-conj-2} should be thought as a quadratic version of the Geisser-Levine Theorem \cite[Thm 8.3]{Geisser_Levine}.
The Geisser-Levine Theorem shows that in the case of a field \(k\) of positive characteristic \(p\), there is no \(p\)-adic motivic cohomology outside of Milnor K-theory; our result, similarly, shows that in characteristic \(2\), there is no Witt motivic cohomology outside Witt K-theory.\\
Let us explain in more detail how to obtain this result assuming Theorem \ref{thm:Morel-str-conj-2} holds.
On the category of motivic spectra \(\SH(k)\), for \(k\) a perfect field of characteristic \(p\), the Geisser-Levine Theorem and Nesterenko-Suslin-Totaro Theorem (\cite{NesterenkoSuslin}, \cite{Totaro_MilnorKtheory}) combine to give a zig-zag of equivalences, for every \(r \geq 1\), 
\begin{equation*}
    \HZ/p^r \overset{\tau_{\leq 0}}{\longrightarrow} \pi_0(\HZ/p^r) \longleftarrow \rmH \K^{M}/p^r,
\end{equation*}
where the first map is the coconnective truncation in the Morel homotopy t-structure, and the second one is the comparison map of Nesterenko-Suslin and Totaro; since \(\HZ\) is also equivalent to the effective cover of \(\rmH \bfK^{M}\) we can also interpret this as stating that, for a base field \(k\) perfect of characteristic \(p\), \(\rmH \bfK^{M}/p^r\) is effective. 
For \(p=2\) and \(r=1\), this implies the bottom map of square \eqref{eq:str-conj-2} is an equivalence.
Since the square \eqref{eq:str-conj-2} is cartesian, the top map must alsobe an equivalence.
Witt motivic cohomology is effective by definition (as it is an effective cover), so the equivalence implies Witt K-theory is also effective. 
Away from characteristic \(2\), Milnor K-theory mod \(2\) is not effective (see Remark \ref{rmk:non-effectivity-KMW}), therefore, this result cannot hold in other characteristics.

Our approach to proving Theorem \ref{thm:Morel-str-conj} and Theorem \ref{thm:Morel-str-conj-2} will actually go in the opposite direction. Assuming Witt K-theory is effective in characteristic \(2\), we will prove that the two squares \eqref{eq:str-conj} and \eqref{eq:str-conj-2} are cartesian. Then by using an effectivity criterion by Bachmann-Fasel \cite[Ch.7, Thm. 4.4]{Bachmann_MWMotives}, we are going to prove that \(\rmH \bfK^{W}\) is effective.

\subsection{Outline}
In section \ref{sec:Preliminaries}, we recall some results about the categories of motivic and effective motivic spectra, and about Milnor-Witt motivic cohomology and Milnor-Witt K-theory.

In section \ref{sec:BachmannProof}, we show how to pass from Bachmann's work to Theorems \ref{thm:Morel-str-conj} and \ref{thm:Morel-str-conj-2}.

In section \ref{sec:Char2}, we complete the proof of Theorems \ref{thm:Morel-str-conj} and \ref{thm:Morel-str-conj-2} in characteristic \(2\), under the assumption that Witt K-theory is effective. Then, using an effectivity criterion by Bachmann-Fasel, we will prove the effectivity of Witt K-theory.

\subsection{Conventions}
We refer to \(\infty\)-categories as ``categories''. We will denote Milnor K-theory as \(K^M_\ast\) for the graded group/ring/sheaf of groups, while we will denote it by \(\bfK^M_\ast\) for the corresponding homotopy module.

\subsection{Acknowledgment}
I am grateful to Elden Elmanto for his guidance and encouragement, and to Tom Bachmann for valuable suggestions on the proof and feedback on a draft of this article. I would also like to thank Klaus Mattis for reading a draft of this article and offering helpful comments.
\newpage
\section{Preliminaries}\label{sec:Preliminaries}

\subsection{SH and homotopy modules}\label{subsec:SH-and-HI}

Let \(k\) be a \emph{perfect} field (throughout all this section).
The category of \emph{motivic spectra} \(\SH(k)\) is the universal stable presentably symmetric monoidal category equipped with a symmetric monoidal functor \(\Sigma^\infty_+: \Smk \to \SH(k)\), satisfying
\begin{itemize}
    \item The canonical map \(\bfA^1_X \to X\) induces an equivalence \(\Sigma^\infty_+(\bfA^1_X) \to \Sigma^\infty_+(X)\).
    \item The Tate motive \(\bfT=\Sigma^\infty (\bfP^1, \infty) \coloneq \mathrm{cofib}(\Sigma^\infty_+ (\Spec(k)) \overset{\infty}{\to} \Sigma^\infty_+(\bfP^1))\) is tensor invertible.
    \item \(\Sigma^\infty_+\) carries Nisnevich distinguished squares in \(\Smk\) to cocartesian squares in \(\SH(k)\). 
\end{itemize}
The category of effective motivic spectra \(\SHeff(k)\) is the smallest stable subcategory of \(\SH(k)\) that is closed under colimits and contains the objects \(\Sigma^\infty_+(X)\) for all \(X \in \Smk\).
We define the category of \(n\)-effective motivic spectra as \(\SHeff(k)(n) \coloneq \SHeff(k) \otimes \bfT^{\otimes n}\). 
These categories allow us to construct the so-called slice filtration on \(\SH(k)\) (originally appearing in \cite[Sec. 2]{Voevodsky_OpenProblems}) as follows.
For all \(n \in \bfZ\), the inclusion \(i^n\) of \(\SHeff(k)(n)\) in \(\SH(k)\) preserves colimits, so there is an adjunction of exact functors
\begin{equation*}
    \SHeff(k)(n) \stackrel[r^n]{i^n}{\longrightleftarrows} \SH(k) \quad \text{with counit} \quad f^n \coloneq  i^n r^n \to \mathrm{id}.
\end{equation*}
Since these counits are the universal maps out of \(n\)-effective motivic spectra, we obtain a map \(f^{n+1} \to f^n\), and we can then consider cofiber sequences \(f^{n+1} \to f^n \to s^n\).
For a motivic spectrum \(E \in \SH(k)\), the slice filtration is given by \(Fil^\ast_{slice}(E) \coloneq  f^\ast(E) \in \mathrm{Fun}(\bfZ^\delta, \SH(k))\).

The functor \(\Sigma^\infty_+: \Smk \to \SH(k)\) factors, by universal properties, through a functor
\begin{equation*}
    \sigma^\infty: \Shv_{Nis, \bfA^1}(\Smk, \mathrm{Sp}) \to \SH(k),
\end{equation*}
where \(\mathrm{Sp}\) is the stable category of spectra.
Furthermore, there is a right adjoint to \(\sigma^\infty\)
\begin{equation*}
    \omega^\infty: \SH(k) \to \Shv_{Nis, \bfA^1}(\Smk, \mathrm{Sp}) , \quad E \mapsto \omega^\infty(E)(-)=\Map_{\SH(k)}(\Sigma^\infty_{+} -, E),
\end{equation*}
where \(\Map\) denotes the mapping spectrum.
There is a graded variant of the adjunction \(\sigma^\infty \dashv \omega^\infty\) given by
\begin{equation*}
   \omega^{\infty, gr} : \SH(k) \leftrightarrows \Shv_{Nis, \bfA^1}(\Smk, \mathrm{Sp})^{\bfZ} : \sigma^{\infty, gr}
\end{equation*}
where 
\begin{equation*}
    \sigma^{\infty, gr}(E)=\bigoplus_{j \in \bfZ} \sigma^{\infty}(E^j) \otimes (\bfT)^{\otimes j}
    \quad \text{and} \quad
    \omega^{\infty, gr}(E)^j = \omega^\infty(E \otimes (\bfT)^{\otimes j}).
\end{equation*}

\begin{lemma}\label{lemma:jointly-conservative}
    The functor \(\omega^{\infty,gr}\) is conservative. 
    The functor \(\omega^{\infty}\)is conservative when restricted to \(\SHeff(k)\). 
\end{lemma}
\begin{proof}
    \(\SH(k)\) is compactly generated (see \cite[Proposition C.12]{Hoyois_QuadraticRefinement}) by the class of objects \(\Sigma^\infty_{+} X \otimes \bfG_m^{j}\) for all \(X \in \Smk\) and for all \(j \in \bfZ\). 
    Similarly, \(\SHeff(k)\) is compactly generated by the class of objects \(\Sigma^\infty_{+} X\). The claims follow. 
\end{proof}
The category \(\SH(k)\) comes with a t-structure \cite[Theorem 5.2.3]{Morel_Introduction}, referred to as the \emph{homotopy t-structure}, whose non-negative part is generated under colimits and  extensions by objects of the form 
\(\Sigma^\infty_{+} X_+\)
(see \cite[Sec. 2.1]{Hoyois_FromAGCtoMC}).\\
For every \(i, j \in \bfZ\) and every motivic spectrum \(E \in \SH(k)\) we can define the \((i,j)\)-homotopy sheaf associated with \(E\) as
\begin{equation*}
    \pi_i(E)_{j}(-) = \pi_i(\omega^\infty(E \otimes \bfG_m^{\otimes j}))(-) \simeq \pi_0(\Map(\Sigma^\infty_+ ( - ) [i], E \otimes \bfG_m^{\otimes j})) \in \Shv_{Nis}(\Smk, \mathrm{Ab}),
\end{equation*}
where \(\pi_0\) denotes the \(0\)-th homotopy group of the mapping spectrum.
Thanks to Morel’s connectivity Theorem \cite[Thm. 3]{Morel_StableConnectiviy}, a more explicit description of the homotopy t-structure is possible by using these homotopy sheaves, as in \cite[Thm. 2.3]{Hoyois_FromAGCtoMC}: \(E \in \SH(k)_{\geq 0}\) if and only if \(\pi_i(E)_j=0\) for all \(i-j < 0\).
From this, it follows that the t-structure is left and right complete \cite[Corollary 2.4]{Hoyois_FromAGCtoMC}. 
\begin{corollary}
    Let \(k\) be a perfect field. The collection of functors \(\pi_i(-)_j: \SH(k) \to \Shv_{Nis}(\Smk, \mathrm{Ab})\) for \(i, j \in \bfZ\) is jointly conservative. The collection of functors \(\pi_i(-)_0: \SHeff(k) \to \Shv_{Nis}(\Smk, \mathrm{Ab})\) for \(i \in \bfZ\) is jointly conservative. 
\end{corollary}
\begin{proof}
    The first statement follows from Lemma \ref{lemma:jointly-conservative} and the fact that \(\Shv_{Nis, \bfA^1}(\Smk, \mathrm{Sp})\) is Postnikov complete. 
    The second statement is \cite[Prop. 4]{Bachmann_GenSlicesKO}, but can also be shown from what we just stated, using the fact that, for \(j \geq 0\), \(\Sigma^\infty_+ X \otimes \bfG_m^{\otimes j}\) is a retract of \(\Sigma^\infty_+( X \times \bfG_m^{\times j})\).
\end{proof}

We can compare \(\pi_0(E)_{\ast} \in \Shv_{Nis}(\Smk, Ab)^\bfZ\) with \(\omega^{\infty, gr}(E) \in \Shv_{Nis}(\Smk, Sp)^\bfZ\), and we should think of this analogously to what happens with abelian groups and spectra.

The heart \(\SH(k)^\heartsuit\) of \(\SH(k)\) is equivalent to the abelian category \(\bfH\bfI(k)\) of homotopy modules over \(k\) \cite[Theorem 5.2.6]{Morel_Introduction}, i.e. \(\bfZ\)-graded strictly \(\bfA^1\)-invariant\footnote{A sheaf of abelian groups \(A\) on \(\Smk\) is called strictly \(\bfA^1\)-invariant if \(H^i_{Nis}(X; A) \overset{\simeq}{\to} H^i_{Nis}(\bfA^1_X; A)\) is an equivalence for every \(X \in \Smk\) and every \(i\).} Nisnevich sheaves of abelian groups
\(M_\ast\) together with isomorphisms \(\mu_n : M_n \overset{\simeq}{\longrightarrow} \Omega_{\Gm}(M_{n+1})\) 
\,\footnote{For a sheaf of abelian groups \(A\), \(\Omega_{\Gm} A(-) \coloneq \mathrm{ker}(A(- \times \bfG_m) \overset{(id \times 1)^\ast}{\to} A(-))\), where \(1\) denotes the inclusion of the point \((z-1) \trianglelefteq k[z,z^{-1}]\) in \(\bfG_m\).}, 
for all \(n \in \bfZ\).
More precisely, there is a functor
\begin{equation*}
    \rmH: \bfH \bfI(k) \to \SH(k), \quad M_{\ast} \to \rmH M
\end{equation*}
which is fully faithful and induces an equivalence with the \(\SH(k)^{\heartsuit}\); its inverse is given by the functor \(E \mapsto \pi_0(E)_{\ast}\). 
Furthermore, by forgetting the bonding maps, we obtain a limit and colimit preserving conservative functor (see \cite[Lemma 1.2]{Bachmann_GenSlicesKO})
\begin{equation*}
    \SH(k)^\heartsuit \to \Shv_{Nis}(\Smk, \mathrm{Ab})^\bfZ, \quad E \mapsto \pi_0(E)_{\ast}.
\end{equation*}
\begin{remark}[\protect{\cite[Lemma 5.2.5]{Morel_Introduction}}]
    The homotopy sheaves \(\pi_i(E)_{\ast}\), for \(E \in \SH(k)\) and for all \(i\), are homotopy modules in a natural way.
\end{remark}
%

\subsection{Unramified Milnor--Witt K-theory}
We are interested in working with unramified cohomology theories\footnote{We are interested in unramified sheaves that are quotients of Milnor--Witt K-theory under suitable relations. Morel, in \cite{Morel_A1top}, called them unramified \(\K^{\bfR}\)-theories.
More precisely, we are interested uniquely in Milnor--Witt, Milnor, and Witt K-theory, and their quotient \(mod \; \ell\) for \(\ell\) prime.} as defined in \cite[Ch. 3.1, 3.2]{Morel_A1top}.
We are not recalling here the definition of unramified sheaves, which can be found in \cite[Ch. 2.1]{Morel_A1top}.
Loosely speaking, unramified sheaves on \(\Smk\) are those objects of \(\Shv_{Nis}(\Smk, \mathrm{Ab})\) whose values are determined by the behavior on codimension \(0\) (i.e., on and on fractions fields of smooth \(k\)-schemes) and codimension \(1\) points (i.e., on DVR that are stalks of the structure sheaf of a smooth scheme).
For \(\bfZ\)-graded sheaf of abelian groups, we will say a sheaf is unramified if all its components are.
For example, any strictly \(\bfA^1\)-invariant sheaf on \(\Smk\) is unramified, so every homotopy module is unramified (\cite[Ex. 3.2]{Morel_A1top}). 
The following definition is due to Hopkins and Morel. 
\begin{definition}[(Milnor--Witt K-theory of a field) \protect{\cite[Definition 3.1]{Morel_A1top}}]\label{def:KMW}
    Let \(F\) be a field.
    Then  the \emph{Milnor--Witt K-theory} of \(F\) the associative graded ring \(K^{MW}_{\ast}(F)\) generated by
    \begin{itemize}
        \item symbols \([u]\) for each unit \(u \in F^\times\), with degree \(1\);
        \item a single symbol \(\eta\), with degree \(-1\);
    \end{itemize}
    subject to relations
    \begin{itemize}
        \item (Steinberg relation) for each \(u \in F\setminus \{0,1\}\): \([u][1-u]=0\);
        \item (Logarithmic relation) for each pair \((u,v) \in (F^\times)^2\): \([uv]=[u]+[v]+\eta[u][v]\);
        \item (Centrality relation) for each \(u \in F^\times\): \([u]\eta=\eta[u]\);
        \item (Hyperbolic relation) let \(h \coloneq  \eta[-1]+2\) be the so-called hyperbolic element: \(\eta h=0\).
    \end{itemize}
\end{definition}

Unramified Milnor--Witt \(\bfK^{MW}_{\ast} \in \Shv_{Nis}(\Smk, \mathrm{Ab})^{\bfZ}\) is a strictly \(\bfA^1\)-invariant\footnote{It is proved in \cite[Lemma 3.37]{Morel_A1top} that \(\K^{MW}_n\) is the initial strongly, and consequently strictly \cite[Thm. 6.10]{bachmann2024stronglya1invariantsheavesafter}, \(\bfA^1\)-invariant sheaf of abelian groups under the sheaf of abelian groups freely generated by the pointed sheaf of sets \(\bfG^{\otimes n}_m\).} unramified sheaf of \(\bfZ\)-graded abelian groups\footnote{It is, in fact, a sheaf of \(\bfZ\)-graded rings.} whose sections on any finitely generated field \(F\) over \(k\) coincide with \(K^{MW}_{\ast}(F)\).
By \cite[Thm. 3.15]{Morel_A1top} (and similarly to \cite{Milnor_AlgebraicKQuadraticForms}), for any discrete valuation \(\nu\) on a field \(F\) (trivial on \(k\)), with valuation ring \(\calO_{\nu} \subset F\), uniformizing element \(\varpi\), and residue field \(\kappa(\nu)\), there exists a unique morphism of graded groups, which we call residue homomorphism, 
\begin{equation*}
    \partial^{\varpi}_{\nu}: K^{MW}_{\ast}(F) \rightsquigarrow K^{MW}_{\ast-1}(\kappa(\nu))
\end{equation*}
which commutes with products by \(\eta\) and satisfies both
\begin{equation*}
    \partial^{\varpi}_{\nu}([\varpi][u_2]\cdots[u_n])=[\overline{u_2}]\cdots[\overline{u_n}]
\end{equation*}
and 
\begin{equation*}
    \partial^{\varpi}_{\nu}([u_1]\cdots[u_n])=0
\end{equation*}
for any units \(u_1, \dots, u_n\) of \(\calO_{\nu}\), where \(\overline{u}_i\) is the residue class of \(u_i\). 
This allows us to define 
\begin{equation}\label{eq:unramifiedness-def-valuations}
    K^{MW}_{\ast}(\calO_{\nu}) \coloneq  \ker (K^{MW}_{\ast}(F) \overset{\partial^{\varpi}_{\nu}}{\rightsquigarrow} K^{MW}_{\ast-1}(\kappa(\nu))),
\end{equation}
and for \(X \in \Smk\) irreducible
\begin{equation}\label{eq:unramifiedness-def-schemes}
    K^{MW}_{\ast}(X)=\cap_{x \in X^{(1)}} K^{MW}_{\ast}(\calO_{X,x}) \hookrightarrow K^{MW}_{\ast}(\mathrm{Frac}(\calO_{X}(X))).
\end{equation}
The residue morphism \(\partial^{\varpi}_{\nu}\) depends on the choice of \(\varpi\).\footnote{In contrast to what happens on Milnor K-theory where the residue homomorphism depends only on the valuation and not on the choice of a uniformizer.}
Indeed, as in \cite[Rmk. 3.20]{Morel_A1top}, for any \(u \in \calO_{\nu}^{\times}\) we get \(\partial^{\varpi}_{\nu}([u \cdot \varpi])=1 + \eta [\overline{u}]\). 
However, by \cite[Lemma 3.19]{Morel_A1top}, the kernel depends exclusively on \(\nu\).
\begin{remark}
    One can also define \emph{symbolic} (as opposed to unramified) Milnor--Witt K-theory as in \cite{Gille_Scully_Zhong} and \cite{bachmann2024stronglya1invariantsheavesafter}, by taking Definition \ref{def:KMW} and making it work for all local rings.
    There is a map from naive Milnor--Witt K-theory to unramified Milnor--Witt K-theory which is an isomorphism not only when taking sections on a field, but also as sheaves when the base field is infinite of characteristic \(\neq 2\).
    The infinitude of \(k\) is known to be necessary, for reasons similar to the Milnor K-theory version of statement \cite{Kerz-Milnor-FiniteResidueFields}.
\end{remark}
\begin{theorem}[(Morel) \cite{Morel_MotivicPi0}, \protect{\cite[Thm. 6.40]{Morel_A1top}}] \label{thm:Motivic_Pi0_SphereSpectrum}
    Let \(k\) be a perfect field and \(\bfS = \Sigma^{\infty}_+ \Spec(k)\) (the tensor-unit of \(\SH(k)\)).
    Then there exists a canonical morphism \(\K^{MW}_{\ast} \to \pi_0(\bfS)_{\ast}\) in \(\Shv_{Nis}(\Smk, \mathrm{Ab})^{\bfZ}\).
\end{theorem}
\begin{corollary}
    Let \(k\) be a perfect field.
    \(\K^{MW}_{\ast}\) is a homotopy module and hence there is a motivic spectrum \(\rmH\K^{MW} \in \SH(k)^{\heartsuit}\) such that 
    and \(\pi_0(\rmH\K^{MW})_{\ast} \simeq \K^{MW}_{\ast}\).
    There is an equivalence  \(\rmH\K^{MW} \overset{\simeq}{\to} \pi_0 \bfS \in \SH(k)^{\heartsuit}\), where \(\pi_0 \bfS\) denotes the truncation \(\tau_{\leq 0} \tau_{\geq 0} \bfS\) in the homotopy t-structure.
\end{corollary}
\begin{remark}\label{rmk:contracting-isos-KMW}
    We can explicitly show that \(\K^{MW}_{\ast}\) is an homotopy module. 
    Let \(k\) be a perfect field.
    As previously mentioned \(\bfK^{MW}_{\ast}\) is a strictly \(\bfA^1\)-invariant sheaf of \(\bfZ\)-graded abelian groups on \(\Smk\).
    We just need to produce an isomorphism (of sheaves of abelian groups) between \(\bfK^{MW}_{n}\) and \(\Omega_{\bfG_m}(\bfK^{MW}_{n+1})\).
    As in \cite[Ch. 2 Lemma 3.1.8]{Bachmann_MWMotives}, for each \(X \in \Smk\), \(i \geq 0\), and \(n \in \bfZ\) we have an isomorphism of \(\bfK^{MW}_0(X)\)-modules
    \begin{equation*}
        H^i_{Nis}(X \times \mathbb{G}_m, \mathbf{K}^{MW}_j) \cong H^i_{Nis}(X, \mathbf{K}^{MW}_j) \oplus H^i_{Nis}(X, \mathbf{K}^{MW}_{j-1}) \cdot [t].
    \end{equation*}
    There is also an exact sequence of modules
    \begin{equation*}
        0 \to \bfK^{MW}_{n+1}(X) \longrightarrow \bfK^{MW}_{n+1}(X)(X \times \bfG_m) \longrightarrow \bfK^{MW}_{n}(X) \to 0
    \end{equation*}
    which has a splitting given by 
    \begin{equation*}
        \bfK^{MW}_{n}(X) \overset{\cdot [t]}{\longrightarrow} \bfK^{MW}_{n+1}(X) \overset{\simeq}{\longrightarrow} \bfK^{MW}_{n+1}(X \times \bfA^1_k) \overset{(\bfG_m \hookrightarrow \bfA^1_k)^\ast}{\longrightarrow}  \bfK^{MW}_{n+1}(X \times \bfG_m).
    \end{equation*}
    It follows that the kernel of the morphism \(\bfK^{MW}_{n+1}(X \times \bfG_m) \overset{(id \times 1)^\ast}{\longrightarrow} \bfK^{MW}_{n+1}(X)\) is exactly \(\bfK^{MW}_{n}(X)\), since \([t]\) is mapped to \([1]\) (and \([1]=0\) by \cite[Lemma 3.5.b]{Morel_A1top}).
\end{remark}
\begin{remark}
    Since every motivic spectrum is a \(\bfS\)-module (as \(\bfS\) is the tensor-unit of \(\SH(k)\)), it follows that every homotopy module (they are all of the form \(\pi_0(E)_{\ast}\) for some motivic spectrum \(E\)) carries a canonical structure of a module over \(\K^{MW}_{\ast}\).
\end{remark}
Following \cite[Ch. 3.2]{Morel_A1top}, we can define homotopy modules of Milnor K-theory and Witt K-theory as quotients of Milnor-Witt K-theory by a set of ``admissible'' relations.
\begin{definition}\label{def:admissible-relations-on-KMW}
    A set \(\calR\) of admissible relations on \(\K^{MW}_{\ast}\) is the datum of graded ideals \(\calR_{\ast}(F)\) of \(K^{MW}_{\ast}(F)\) for each finitely generated field \(F\) over \(k\) satisfying the following:
    \begin{itemize}
        \item For any extension \(F/E\) of finitely generated fields over \(k\), the map \(K^{MW}_{\ast}(E) \to K^{MW}_{\ast}(F)\) maps  \(\calR_{\ast}(E)\) into \(\calR_{\ast}(F)\).
        \item For any valuation \(\nu\) on a finitely generated field \(F\) over \(k\), and any uniformizer element \(\varpi\), the residue morphism \(\partial^{\varpi}_{\nu}\) maps \(\calR_{\ast}(F)\) into \(\calR_{\ast}(\kappa(\nu))\).
        \item For any finitely generated field \(F\) over \(k\), there is a short exact sequence 
            \begin{equation*}
                0 \longrightarrow \calR_{\ast}(F) \longrightarrow \calR_{\ast}(F(T)) \overset{\sum \partial^P_{D_P}}{\longrightarrow} \oplus \calR_{\ast-1}(F(T)/P) \longrightarrow 0.
            \end{equation*}
    \end{itemize}
\end{definition}
\noindent
By \cite[Thm. 2.46]{Morel_A1top}, these conditions are enough to obtain an homotopy module \(\K^{MW}_{\ast}/\calR_{\ast}\) whose value on a field \(F\) is given by \(K^{MW}_{\ast}(F)/\calR_{\ast}(F)\).
The next lemma is enough to obtain all kinds of admissible relations we are interested in. 
\begin{lemma}[\protect{\cite[Thm. 2.46, Lemma 3.32]{Morel_A1top}}]\label{lemma:quotient-by-admissible-relations}
    Let \(k\) be a perfect field.
    Let \(R_{\ast}\) be a graded ideal of \(K^{MW}_{\ast}(k)\).
    For any field \(F\) finitely generated over \(k\), we denote by \(\calR_{\ast}(F)\) the graded ideal \(R_{\ast} \cdot K^{MW}(F)\) of \(K^{MW}(F)\).
    \(R_{\ast}(F)\) is an admissible set of relations on \(K^{MW}(F)\). 
    In particular, there is an homotopy module \(K^{MW}_{\ast}/\calR_{\ast}\) whose value on F is given by \(K^{MW}_{\ast}(F)/\calR_{\ast}(F)\).
\end{lemma}
\begin{definition}
    Let \(k\) be a perfect field.
    We define the Milnor K-theory and Witt K-theory homotopy modules using Lemma \ref{lemma:quotient-by-admissible-relations} as quotients of Milnor--Witt K-theory \(\bfK^{MW}_{\ast}\):
    \begin{equation*}
       \K^{M}_{\ast} \coloneq  \K^{MW}/(\eta)
       \quad \text{and} \quad
        \K^{W}_{\ast} \coloneq \K^{MW}/(h),
    \end{equation*}
    where \((\eta)\) and \((h)\) are graded ideals of \(K^{MW}_{\ast}(k)\) generated by the corresponding element.
    Using the functor \(\rmH: \bfH \bfI(k) \to \SH(k)\), we obtain motivic spectra
    \begin{equation*}
        \rmH \bfK^M \quad \text{and} \quad \rmH \bfK^{W} \in \SH(k)^{\heartsuit}.
    \end{equation*}
\end{definition}
\begin{remark}
    For a finitely generated fields \(F\) over \(k\), the value of \(\K^{M}_{\ast}(F)= K^{MW}_{\ast}(F)/(\eta)\) coincides with the classical definition of Milnor K-groups as defined in \cite{Milnor_AlgebraicKQuadraticForms}, i.e., \(\mathrm{Tens}_\ast(F^\times)/ (a \otimes (1-a) \colon a \in F \setminus \{0,1\})\). 
\end{remark}
\noindent For a finitely generated field \(F\) over \(k\), we have maps of \(\bfZ\)-graded algebras 
\begin{equation*}
q_M(F): K^{MW}_\ast(F) \to K^{M}_\ast(F) \quad \text{and} \quad q_W(F): K^{MW}_\ast(F) \to K^{W}_\ast(F),
\end{equation*}
called \emph{forgetful maps}, which are the quotient maps by \((\eta)\) and \((h)\), respectively.
We want to lift these maps to maps of unramified sheaves, i.e., natural transformations of the underlying functors.
We do this for \(q_M\) using the fact that \(\bfK^{MW}_\ast\) is unramified; the case for \(q_W\) is the same.
Under the conditions of \eqref{eq:unramifiedness-def-valuations}, there is a map \(q_M(\calO_\nu): \bfK^{MW}_{\ast}(\calO_{\nu}) \to \bfK^{M}_{\ast}(\calO_{\nu}) \) simply because these two \(\bfZ\)-graded algebras are defined as kernels.
Under the conditions of \eqref{eq:unramifiedness-def-schemes}, the map \(q_M(X): \bfK^{MW}_{\ast}(X) \to \bfK^{M}_{\ast}(X)\) exists because the intersection in \eqref{eq:unramifiedness-def-schemes} is a limit condition.
These natural transformations \(q_M\) and \(q_W\) give maps of motivic spectra
    \begin{equation*}
        \mathfrak{q}_M \coloneq \rmH(q_M): \rmH \bfK^{MW} \to \rmH \bfK^{M} \quad \text{and} \quad \mathfrak{q}_W \coloneq \rmH(q_W): \rmH \bfK^{MW} \to \rmH \bfK^{W}.
    \end{equation*}
    These are the maps appearing in Theorem \ref{thm:Morel-str-conj} and Theorem \ref{thm:Morel-str-conj-2}, respectively.   
\begin{remark}\label{rmk:contracting-isos-KM-and-KW}
    Similarly to Remark \ref{rmk:contracting-isos-KMW}, we can construct the contracting isomorphisms explicitly for \(\bfK^{M}_{\ast}\) and \(\bfK^{W}_{\ast}\) using Bass-Tate-type sequences.
\end{remark}

Now that we have constructed our these various version of K-theor, we can construct various versions of  motivic cohomology.
\begin{definition}
    We define the following objects in \(\SH(k)\), representing different variants of motivic cohomology, as motivic spectra:
    \begin{itemize}
        \item \emph{Motivic Cohomology} 
            \begin{equation*}
                \HZ \coloneq s^0(\bfS).
            \end{equation*}
        \item \emph{Milnor-Witt Motivic Cohomology}
            \begin{equation*}
                \HZtilde \coloneq f^0(\rmH\K^{MW}).
            \end{equation*}
        \item \emph{Witt Motivic Cohomology}
            \begin{equation*}
                \HWZ \coloneq f^0(\rmH\K^{W}).
            \end{equation*}
    \end{itemize}
\end{definition}
\begin{remark}[(On the definition of motivic cohomology)]
    There are equivalences \(\HZ=s^0(\bfS) \to s^0(\rmH \K^{M})\) and  \(f^0(\rmH \K^{M}) \to s^0(\rmH \K^{M})\), where the first map is given by applying the \(0\)-effective slice to the composition \(\bfS \to \pi_0(\bfS)_{\ast} \overset{\simeq}{\longrightarrow} \K^{MW}_{\ast} \to \K^{M}_{\ast}\) and the second map is the canonical map of the cofiber sequence \(f^1 \to f^0 \to s^0\).
    A proof of this can be found, for example, in \cite[Lemma 12]{Bachmann_GenSlicesKO}.\\
    The map \(\bfS \to \bfK\bfG\bfL\) induces an equivalence after applying \(s^0\) (due to \cite{Levine_Coniveau}[Sections 6.4 and 9]).
    Our definition \emph{agrees} with the classical definition by Voevodsky in \cite[Thm. 19.1]{Voevodsky_LectureNotes}.
    By \cite[Thm. 6.5.1, Thm. 9.0.3]{Levine_Coniveau}, our definition \emph{agrees} with the definition in terms of Bloch's higher Chow groups.
\end{remark}
\begin{definition}
     Let \(n \geq 1\). We define the \emph{mod \(n\) Milnor} 
     \(K\)-theory motivic spectrum as the cofiber in \(\SH(k)\)
     \begin{equation*}
         \rmH \bfK^{M}/n \coloneq \mathrm{cof}(\rmH \bfK^{M} \overset{\cdot n}{\longrightarrow} \rmH \bfK^{M}).
     \end{equation*}
     Similarly, we define the \emph{mod \(n\) Motivic Cohomology} as the cofiber 
     \begin{equation*}
         \HZ/n \coloneq \mathrm{cof}(\HZ \overset{\cdot n}{\longrightarrow} \HZ).
     \end{equation*}
\end{definition}
\begin{remark}
    \(f^0\) is an exact, colimit-preserving functor; applying it to the cofiber sequence \(\rmH \K^{M} \overset{n \cdot}{\to }\rmH \K^{M} \to \rmH \K^{M} /n\), we obtain an equivalence
    \begin{equation*}
        f^0(\rmH \K^{M}/n) \to f^0(\rmH \K^{M})/n = \HZ/n
    \end{equation*}
    for every \(n \geq 1\).

\end{remark}
\begin{remark}
    There is another notion of mod \(n\) Milnor \(K\)-theory, which is given by taking the cokernel in the abelian category \(\SH(k)^{\heartsuit}\), instead of the cofibers in \(\SH(k)\). The construction can be done also by taking \((\eta, n)\) as a set of admissible relations.\\
    By \cite{Izhboldin}[Thm. A], for any field \(F\) of characteristic \(p\), \(K^M_\ast(F)\) is \(p\)-torsion free. 
    In particular, the map \(K^M_\ast(F) \overset{\cdot p}{\longrightarrow} K^M_\ast(F)\) is injective. 
    If we fix \(k\) perfect field of characteristic \(p\), since the map is injective on fields, the map on homotopy modules \(\rmH\bfK^{M}\overset{\cdot p}{\longrightarrow} \rmH\bfK^{M}\) has trivial kernel. Consequently, \(\rmH\bfK^{M}/p\) is both the cokernel in \(\SH(k)^{\heartsuit}\) and the cofiber in \(\SH(k)\).
\end{remark}
\begin{remark}\label{rmk:connectiveness-motivic-cohomology}
    By \cite[Prop. 4]{Bachmann_GenSlicesKO} the functor \(f^0\) is right t-exact for the homotopy t-structure.
    Thus, all our motivic cohomologies are connective for the Morel's t-structure.
\end{remark}
\begin{theorem}[(Nesterenko-Suslin-Totaro) \protect{\cite{NesterenkoSuslin}, \cite{Totaro_MilnorKtheory}}]\label{thm:MilnorLine}
    Let \(F\) be a field.
    There is an equivalence of graded rings \(K^{M}_\ast(F) \to H^\ast(F, \bfZ(\ast))\).
\end{theorem}
\begin{corollary}[(Nesterenko-Suslin-Totaro for motivic spectra) \protect{\cite[Thm. 5.3.2]{Morel_Introduction}}]
    Let \(k\) be a perfect field. There is an equivalence of homotopy sheaves \(\bfK^{M}_{\ast} \to \pi_0(\HZ)_{\ast} \in \Shv_{Nis}(\Smk, \mathrm{Ab})^{\bfZ}\) and so an equivalence of motivic spectra \(\rmH\bfK^{M} \to \pi_0(\HZ) \in \SH(k)\).
\end{corollary}

\subsection{Effective objects and Geisser-Levine Phenomena}

For every connective motivic spectrum \(E \in \SH(k)_{\geq 0}\) there is a map \(E \to \tau_{\leq 0}(E)=\pi_0(E)\).
Even if \(E\) is effective, \(\pi_0(E)\) does not have to be.
For example, if \(E=\bfS\) (the sphere spectrum), then \(\rmH \bfK^{MW} \overset{\simeq}{\to} \pi_0(\bfS)\) is the Milnor-Witt K-theory motivic spectrum (see Theorem \ref{thm:Motivic_Pi0_SphereSpectrum}); and if \(E=\HZ\) is motivic cohomology, then \(\rmH \bfK^{M} \overset{\simeq}{\to} \pi_0(\HZ)\) is Milnor K-theory motivic spectrum (see Theorem \ref{thm:MilnorLine}).
In general, Milnor-Witt and Milnor K-theory are not effective because of the following.

\begin{remark}[(Non effectivity of \(\rmH\K^{M}\) and \(\rmH\K^{MW}\))]\label{rmk:non-effectivity-KMW}
    Let \(k\) be a perfect field of any characteristic, \(\ell\) be a prime invertible in \(k\), and assume \(k\) contains the \(l\)-adic roots of unity.
    By \cite[Lemma 2.4]{Orlov_Vishik_Voevodsky} there is a cofiber sequence in \(\SH(k)\) 
    \begin{equation}\label{eq:coconnective-cover-HZ}
        \HZ/\ell[1] \otimes \; \bfG_m^{\otimes -1} \to \HZ/\ell \to \rmH \K^{M}/\ell. 
    \end{equation}
    If \(\rmH \K^{M}/\ell\) were effective, then the cofiber sequence would force also \(\HZ/\ell[1] \otimes \; \bfG_m^{\otimes -1}\) to be, but it is evidently not due to de-suspension. 
    Therefore \(\rmH\K^{M}\) cannot be effective, and the same conclusion holds for \(\rmH \K^{MW}\), otherwise effectivity would descend to the quotient \(\K^{M}\) as well.\\
    We should interpret this cofiber sequence as saying that to kill the positive part of the t-structure of \(\HZ/\ell\) we need to attach some cells which, however, are not effective, making it impossible for the coconnective cover to be effective. 
\end{remark}

The Geisser-Levine Theorem shows that in the case of a field \(k\) of positive characteristic \(p\), there is no p-adic motivic cohomology outside of Milnor K-theory.
\begin{theorem}[(Geisser-Levine) \protect{\cite[Thm. 8.4]{Geisser_Levine}}]\label{thm:Geisser-Levine}
    Let \(F\) be a field of characteristic \(p \neq 0\) then the top degree morphism \(\bfZ(j)(F) \to H^j(F, \bfZ(j))[-j] \) induces an equivalence after taking quotient modulo \(p^r\)
\begin{equation*}
    \bfZ(j)(F)/p^r \to H^j(F, \bfZ(j))[-j]/p^r.
\end{equation*}
\end{theorem}
At the level of motivic spectra we can express the result of the Geisser-Levine Theorem as in the following.
\begin{corollary}\label{cor:Geisser-Levine-motivic}
    If \(k\) is a perfect field of characteristic \(p \neq 0\), then the morphism \(\HZ/p^r \to \pi_0(\HZ/p^r)\) is an equivalence for every \(r \geq 1\). 
\end{corollary}
\begin{proof}
    Assume Theorem \ref{thm:Geisser-Levine}. 
    By the homotopy t-structure we have a natural transformation of functors \(\SH(k) \to \SH(k)\) given by \(id \to \tau_{\leq 0}\), the coconnective cover.
    Since \(\HZ\) is connective, also \(\HZ/p^r \coloneq \mathrm{cof}(\HZ \overset{p^r}{\to} \HZ)\) is connective.
    It follows that \(\tau_{\leq 0} \HZ/p^r \overset{\simeq}{\to} \pi_0 \HZ/p^r\).
    So we can consider the morphism in \(\SH(k)\)
    \begin{equation*}
        \HZ/p^r \to \pi_0(\HZ/p^r).
    \end{equation*}
    Applying \(\omega^\infty( - \otimes \bfG_m^{\otimes j})\), we obtain morphisms from
    \begin{equation*}
        \omega^\infty {(\HZ/p^r \otimes \bfG_m^{\otimes j})} 
        \overset{\simeq}{\longrightarrow}
        \omega^\infty {((\HZ \otimes \bfG_m^{\otimes j})/p^r)}  
        \overset{\simeq}{\longrightarrow}
        \omega^\infty {(\HZ \otimes \bfG_m^{\otimes j})}/p^r
        \overset{\simeq}{\longrightarrow}
        \bfZ(j)/p^r
    \end{equation*}
    to 
    \begin{equation*}
        \omega^\infty {(\pi_0(\HZ/p^r) \otimes \bfG_m^{\otimes j})}
        \overset{\simeq}{\longrightarrow}
        \omega^\infty {(\pi_0(\HZ)/p^r \otimes \bfG_m^{\otimes j})}
        \overset{\simeq}{\longrightarrow}
        \omega^\infty {(\pi_0(\HZ) \otimes \bfG_m^{\otimes j})}/p^r.
    \end{equation*}
    Thus by Theorem \ref{thm:Geisser-Levine}, Remark \ref{rmk:fields} by the conservativity of \(\omega^{\infty,gr}\) (Lemma \ref{lemma:jointly-conservative}), the equivalence follows.
\end{proof}

\begin{corollary}
   Let \(k\) be a perfect field of characteristic \(p \neq 0\).
   There are equivalences
    \begin{equation*}
        \HZ/p^r \overset{\simeq}{\longrightarrow} \pi_0(\HZ/p^r) \overset{\simeq}{\longleftarrow} \rmH \K^{M}/p^r.
    \end{equation*}
    In particular, \(\rmH \K^{M}/p^r\) is effective for every \(r \geq 1\). 
\end{corollary}

Using this, we immediately see that, if \(k\) is a perfect field of characteristic \(2\) the lower arrow in the square \eqref{eq:str-conj-2} is an equivalence, and so must be the upper one. We can interpret this as the fact that in characteristic \(2\) there is no Witt motivic cohomology outside the Witt K-theory, so that \(\rmH \K^{W}\) is effective. This result should be interpreted as a quadratically enriched version of Geisser-Levine.

\begin{remark}[(Non effectivity of \(\rmH\K^{W}\) in characteristic \(\neq 2\))]\label{rmk:non-effectivity-KW}
    If \(k\) is a perfect field of characteristic \(\neq 2\), then \(\rmH\K^{W}\) cannot be effective.
    Indeed, assuming Theorem \ref{thm:Morel-str-conj-2}, if \(\rmH\K^{W}\) were effective, then the upper horizontal arrow would have been an equivalence.
    Since the square \eqref{eq:str-conj-2} is cartesian and \(\SH(k)\) is stable, the lower horizontal map \(\HZ/2 \to \rmH\bfK^{M}_2\) would be an equivalence; however, in remark \ref{rmk:non-effectivity-KMW}, we have shown that \(\rmH\K^{M}/2\) is not effective, so this it cannot be equivalent to \(\HZ/2\). 
    By reducing the square modulo a prime \(\ell \neq 2\) or by tensoring with \(\bfQ\) we obtain cartesian squares 
    \begin{equation*}
        \begin{tikzcd}[cramped,sep=scriptsize]
        	{\HWZ / \ell} && {\rmH\K^{W}/\ell} && {\HWZ \otimes \; \bfQ } && {\rmH\K^{W}\otimes \; \bfQ } \\
        	&&& {\text{and}} \\
        	{(\HZ/2)/ \ell \simeq 0} && {(\rmH\K^M/2)/ \ell \simeq 0} && {(\HZ/2)\otimes \; \bfQ \simeq 0} && {(\rmH\K^M/2) \otimes \; \bfQ \simeq 0}
        	\arrow[from=1-1, to=1-3]
        	\arrow[from=1-1, to=3-1]
        	\arrow[from=1-3, to=3-3]
        	\arrow[from=1-5, to=1-7]
        	\arrow[from=1-5, to=3-5]
        	\arrow[from=1-7, to=3-7]
        	\arrow[from=3-1, to=3-3]
        	\arrow[from=3-5, to=3-7]
        \end{tikzcd}
    \end{equation*}
    which guarantees \(\rmH\K^{W}/\ell\) and \(\rmH\K^{W} \otimes \; \bfQ\) are effective.
    Therefore the failure of effectivity is concentrated in the mod \(2\) case\footnote{See remark \ref{rmk:smaller-assumptions} for an explanation of the meaning of this.}.
\end{remark}

To prove Theorem \ref{thm:Morel-str-conj} and Theorem \ref{thm:Morel-str-conj-2} in characteristic \(2\) we will need to prove that \(\rmH \K^{W}\) is effective in characteristic \(2\).
This will be verified by means of the following effectivity criterion.

\begin{theorem}[(Bachmann-Fasel effectivity Criterion)\protect{\cite[Ch.7, Thm. 4.4]{Bachmann_MWMotives}}]\label{thm:Bachmann-Fasel-effectivity}
     Let \(k\) be a perfect field and \(E \in \SH(k)\).
     Then \(E\) is effective if and only if for all \(n \geq 1\) and all finitely generated fields \(F/k\), it holds 
     \begin{equation*}
         \left\lvert \omega^\infty(E)(n)( \hat{\Delta}^\bullet_F ) \right\rvert \simeq \ast.
     \end{equation*}
\end{theorem}
\noindent
We need to explain some of the objects used in this theorem.
Let \(F\) be a field; \(\hat{\Delta}^n_F\) denotes the semi-localization
\footnote{Semi-localization for schemes is the analog of localization but for a finite collection of points instead of a single one.
This procedure does not behave well for all schemes, but it does for affine finite schemes, i.e., a scheme for which every finite subset is contained in an affine open subset; for example, any quasi-projective scheme over an affine scheme is affine finite \cite[Prop. 3.3.36]{Liu_Algebraic_Geomtry_Arithmetic_Curves}. 
A scheme is semi-local if it is the localization of a smooth integral scheme at finitely many points.} 
of the algebraic \(n\)-simplex
\begin{equation*}
     \Delta^n_F \coloneq \Spec\left(F[x_0,\dots,x_n]/\left(\sum_{i=0}^n x_i=1\right)\right)
\end{equation*}
at the n+1 vertices \(\{v_0, \dots v_n\}\) of \(\Delta^n_F\), i.e., 
\begin{equation*}
    \hat{\Delta}^n_F= \lim_{\{v_0, \dots, v_n\} \subset U \subset \Delta^n_F} U,
\end{equation*}
where the limit runs over the open subsets of \(\Delta^n_F\) containing all the vertices. 

In Theorem \ref{thm:Bachmann-Fasel-effectivity} we are interested in computing the value of \(\omega^\infty(E)(n)\) on the semi-local scheme \(\hat{\Delta}^\bullet_F\).
However, \(\hat{\Delta}^\bullet_F\) does not have to be a smooth \(k\)-scheme. 
We are interested in extending the definition of \(\omega^\infty(E)(n)\) from the category of smooth \(k\)-schemes \(\Smk\) to a larger category of schemes \(\essSmk\) called essentially smooth \(k\)-schemes.
An essentially smooth \(k\)-scheme is a scheme \(X\) that is the inverse limit \(\varprojlim X_{\alpha}\) of smooth \(k\)-schemes \(X_\alpha\) along  \'etale affine morphisms \(X_{\alpha} \to X_{\beta}\).
\footnote{Not all essentially smooth schemes are noetherian.}
For example, if \(F\) is a finitely generated field extension of a perfect field \(k\), then \(\Spec(F)\) is essentially \(k\)-smooth \cite[Lemma A.2]{Hoyois_FromAGCtoMC}. 
If \(X\) is a smooth \(k\)-scheme and \(x\) is a point in \(X\), then the local scheme \(X_x \coloneq \Spec(\calO_{X,x})\), its henselization, and its strict henselization are essentially \(k\)-smooth.\\
We can extend \(F: \Smk^{op} \to \mathrm{Sp}\) from smooth \(k\)-schemes to a functor \(\tilde{F}\) of essentially-smooth \(k\)-schemes by left Kan extension. 
This notion is well defined by \cite[Proposition 8.14.2]{EGA4}.
So we can extend
\(\omega^\infty(E) \in \Shv_{Nis, \bfA^1}(\Smk, \mathrm{Sp})\) to
\begin{equation*}
    \omega^\infty(E) \in \Shv_{Nis, \bfA^1}(E\Smk, \mathrm{Sp})
\end{equation*}
and write \(\omega^\infty(E)(\hat{\Delta}^n_F)\) for
\begin{equation*}
    \colim_{\{v_0, \dots, v_n\} \subset U \subset \Delta^n_F} \omega^\infty(E)(U).
\end{equation*}

We will also need to rewrite the condition of Theorem \ref{thm:Bachmann-Fasel-effectivity} in terms of \(0\)-slices, which we can do thanks to the following.
\begin{theorem}[(Homotopy Coniveau Tower) \protect{\cite{Levine_Coniveau}, \cite[Ch.7, Lemma 4.3]{Bachmann_MWMotives}}]\label{thm:h-coniveau}
    Let \(k\) be a perfect field, \(E \in \SH(k)\), and \(F\) a finitely generated extension of \(k\).
    Then \(\omega^{\infty}(s^0 E)(F) \to \left\lvert \omega^{\infty}(E)( \hat{\Delta}^\bullet_F ) \right\rvert\) is an equivalence.
\end{theorem}

\subsection{Grothendieck--Witt Theories}
Later, we will apply the Bachmann-Fasel effectivity criterion \ref{thm:Bachmann-Fasel-effectivity} to prove  \(\rmH \K^{W}\) is effective.
To do this, we need a more explicit description of \(\rmH \K^{W}\).

\begin{definition}
    The \emph{Grothendieck--Witt ring} of \(F\) is the commutative ring \(GW(F)\) defined as the Gro\-thendieck group of the commutative monoid of isometry classes of nondegenerate symmetric bilinear forms\footnote{If the characteristic of \(F\) is \(2\), there is also a different definition using quadratic form.} over \(F\) with orthogonal sum as addition, and tensor product of forms as multiplication.  
\end{definition}
For \(u \in F^\times\), we denote by \(\langle u \rangle\) the inner product space of dimension one (hence automatically symmetric) such that, if \(e\) is a basis for this space, then \(\langle e, e \rangle =u\). 
A specific example of a non-degenerate symmetric bilinear form which is very important for us is the \emph{hyperbolic form}.
It is defined as the pair \((F^2, h)\), where \(h\) is given by
\begin{equation*}
    h : (F \oplus F) \otimes (F \oplus F) \longrightarrow F,
    \qquad
    h( x \oplus y, x' \oplus y') = xy' + x'y,
\end{equation*}
or equivalently \(h \coloneq  1+ \langle -1 \rangle\).\\
The Grothendieck-Witt ring also has a symbolic description, illustrated for example in \cite[Thm 2.8]{carlier2023milnorwittktheorywittktheory}.
\begin{definition}
    The \emph{Witt ring} of \(F\) is defined as \(W(F)=GW(F)/\langle h\rangle\), where \(\langle h\rangle\) represents the ideal of \(GW(F)\) generated by \(h\) \footnote{Thanks to the symbolic description of \(GW(F)\), the subgroup generated by \(h\) is already an ideal.}.
\end{definition}
\noindent This means that two forms represent the same class in \(W(F)\) if their difference is hyperbolic.

The Grothendieck--Witt ring \(GW(F)\) admits a (surjective) \emph{rank morphism} \(\mathrm{rk}: GW(F) \to \bfZ\), and the Witt ring \(W(F)\) admits a corresponding (surjective) \emph{rank mod \(2\)} morphism \(\mathrm{rk}: W(F) \to \bfZ/2\).
The Witt ring has maximal ideal \(I(F) \coloneq  \ker\big(\mathrm{rk} : W(F) \to \bfZ/2\big)\) called the \emph{fundamental ideal}.
If \(F\) is not formally real, then \(W(F)\) is a local ring of Krull dimension \(0\) and \(I(F)\) is the maximal ideal, see \cite[Prop. 31.4]{Elman_Karpenko_Merkurjev}.
We can form a graded \(W(F)\)--algebra \(I^\ast(F)=I^n(F)\), where for \(n\) positive \(I^n(F)\) denotes the \(n\)-th power of the fundamental ideal, and for \(n\) non-positive \(I^n(F)= W(F)\).

\begin{remark}
    Let \(k\) be a perfect field.
    By \cite[Example 3.34]{Morel_A1top} there is an homotopy module of powers of the fundamental ideal \(\bfI^\ast\), and so a motivic spectrum \(H\bfI \in \SH(k)\).
\end{remark}

\begin{theorem}[(Morel, Carlier)  \protect{\cite[Rmk. 3.12]{Morel_A1top}}, \protect{\cite[Thm. 3.12]{carlier2023milnorwittktheorywittktheory}}]\label{thm:description-of-KW}
    Let \(F\) be any field.
    The unique map of \(W(F)\)-graded algebras \(K^{W}_{\ast}(F) \to I^{\ast}(F)\) that sends 
    \begin{equation*}
        [[a]] \mapsto \langle \langle a \rangle \rangle \quad \text{and} \quad \eta \mapsto -1 \in I^{-1}(F)=W(F) 
    \end{equation*}
    is an isomorphism. 
\end{theorem}

\begin{corollary}
    Let \(k\) be a perfect field. 
    Since homotopy modules are unramified, the isomorphism of Theorem \ref{thm:description-of-KW} lifts to an isomorphism of homotopy modules \(\bfK^{W}_{\ast} \to \bfI^{\ast} \in  \Shv_{Nis}(\Smk, \mathrm{Ab})^{\bfZ}\), and consequently to an equivalence of motivic spectra \(\rmH\bfK^{W}\to \rmH\bfI \in \SH(k)\).
\end{corollary}

\begin{theorem}[(Kato, Orlov-Vishik-Voevodsky) \protect{\cite[Thm. 2.2.3]{deglise2025notesmilnorwittktheory}}]\label{thm:quotient-of-fundamental-ideal}
    Let \(F\) be any field. The map 
    \begin{equation*}
        F^\times \longrightarrow I(F), \quad u \mapsto \langle \langle u \rangle \rangle
    \end{equation*}
    induces a ring isomorphism
     \begin{equation*}
         K^M_\ast(F)/2 \longrightarrow \bigoplus_{d \geq 0} I^d(F)/I^{d+1}(F).
     \end{equation*}
\end{theorem}
This implies there is a short exact sequence of \(\bfZ\)-graded abelian groups
\begin{equation*}
    0 \longrightarrow I^{\ast+1}(F) \longrightarrow I^{\ast}(F) \longrightarrow \K^M_{\ast}(F)/2 \longrightarrow 0.
\end{equation*}
This sequence also exists at the level of motivic spectra.
\begin{corollary}
    Let \(k\) be a perfect field. There is a cofiber sequence in \(\SH(k)\) 
    \begin{equation}\label{eq:short-exact-seq-Kato}
        \rmH\bfK^W \otimes \, \bfG_m \longrightarrow \rmH\bfK^W \longrightarrow \rmH\bfK^M/2.
    \end{equation}
\end{corollary}
\begin{proof}
    To prove the statement in \(\SH(k)^{\heartsuit}\) we just need to verify that \(\pi_0(\rmH \K^{W} \otimes \, \bfG_m)_{\ast}\) is the kernel of the induced map 
    \(\pi_0(\rmH\bfK^W)_{\ast} \to \pi_0(\rmH\bfK^M/2)_{\ast}\). We can check this on henselian affine ind-smooth \(k\)-schemes, and we can reduce checking it on fields. For these, we just have to notice that \(\pi_0(\rmH \K^{W})_{\ast+1} \overset{\simeq}{\to}\pi_0(\rmH \K^{W} \otimes \, \bfG_m)_{\ast}\) is an equivalence, and use Theorem \ref{thm:description-of-KW} and Theorem \ref{thm:quotient-of-fundamental-ideal}. 
    
    Next, consider the fiber \(P\) of the map \( \rmH\bfK^W \longrightarrow \rmH\bfK^M/2\) in \(\SH(k)\). By considering  the long exact sequence, we get
    \begin{gather*}
        \cdots \to 0 \to \pi_0(P) \to \pi_0(\rmH\bfK^{W}) \to \pi_0(\rmH\bfK^{M}/2)  \to \pi_{-1}(P) \to 0 \to \cdots
    \end{gather*}
    which vanishes outside the illustrated part. But this sequences is 
    \begin{gather*}
         0\to \pi_0(P) \to \rmH\bfK^{W} \to \rmH\bfK^{M}/2  \to \pi_{-1}(P) \to 0 
    \end{gather*}
    and because \(\mathrm{coker}(\rmH\bfK^{W} \to \rmH\bfK^{M}/2) \simeq 0\), we get that \(\pi_{-1}(P)\) must vanish. Therefore, the map \(P \to \rmH\bfK^{MW}\) is an equivalence.
\end{proof}

\subsection{A square involving K-theories}
A key result for us is the following theorem, which offers a description of Milnor-Witt K-theory in terms of Milnor and Witt K-theory.
\begin{theorem}[(Kato, Carlier in char. 2) \cite{Kato82}, \protect{\cite[Corollary 2.3.7]{deglise2025notesmilnorwittktheory}}, \cite{carlier2023milnorwittktheorywittktheory}]\label{thm:squareKMW-Algebras}
    Let \(F\) be any field. Then there is a cartesian square of \(\bfZ\)-graded algebras
    \begin{equation*}
        \begin{tikzcd}
        	{\K^{MW}_{\ast}(F)} &&& {\K^{W}_{\ast}(F)} \\
        	\\
        	{\K^{M}_{\ast}(F)} &&& {\K^{M}_{\ast}(F)/2.}
        	\arrow[from=1-1, to=1-4]
        	\arrow[from=1-1, to=3-1]
        	\arrow[from=1-4, to=3-4]
        	\arrow[from=3-1, to=3-4]
        \end{tikzcd}
    \end{equation*}
\end{theorem}

\begin{corollary}\label{cor:squareKMW-Algebras}
    Let \(k\) be a perfect field.  Then there is a cartesian square in \(\SH(k)\)
    \begin{equation}\label{eq:right-square}
        \begin{tikzcd}
        	{\rmH\K^{MW}} &&& {\rmH\K^{W}} \\
        	\\
        	{\rmH\K^{M}} &&& {\rmH\K^{M}/2.}
        	\arrow[from=1-1, to=1-4]
        	\arrow[from=1-1, to=3-1]
        	\arrow[from=1-4, to=3-4]
        	\arrow[from=3-1, to=3-4]
        \end{tikzcd}
    \end{equation}
\end{corollary}
\begin{proof}
    We start by constructing the square in \(\SH(k)^\heartsuit\). Recall from \cite[Lemma 1.2]{Bachmann_GenSlicesKO} there is  a limit and colimit preserving conservative map
    \begin{equation*}
        \sigma: \SH(k)^\heartsuit \to \mathrm{Ab}(\Shv_{Nis}(\Smk, \Ani)_{\leq 0})^\bfZ=(\Shv_{Nis}(\Smk, \mathrm{Ab}))^{\bfZ}, \quad E \mapsto \pi_0(E)_{\ast}.
    \end{equation*}
    Therefore, it is enough to prove the statement at the level of sheaves. 
    Since we are working with Nisnevich sheaves on \(\Smk\) it is enough to show the equivalence when valued on henselian local ind-smooth \(k\)-schemes (the points of the Nisnevich topology on  \(\Smk\)).
    Let \(R\) be a henselian ind-smooth (over \(k\)) local ring. Recall that we defined
    \begin{equation*}
        \K^{MW}_{\ast}(R) = \bigcap_{p \in R^{(1)}} \K^{MW}_{\ast}(R_p) \subset \K^{MW}_{\ast}(Frac(R)),
    \end{equation*}
    where the intersection is taken on height-one prime ideals, and
    \begin{equation*}
        \K^{MW}_{\ast}(R_p) = \ker (\K^{MW}_{\ast}(Frac(R)) \overset{\partial}{\longrightarrow} \K^{MW}_{\ast-1}(Frac(R/p)))
    \end{equation*}
    is the kernel of the residue morphism
    \footnote{We are writing \(\partial\) instead of \(\partial^{\varpi}_{\nu}\), but the morphism itself is really dependent on the choice of a uniformizer for \(R_p\).}.
    All the sheaves we are dealing with are defined by the first two terms of the Rost-Schmid complex. 
    \footnotesize
    \begin{equation*}
        \begin{tikzcd}[cramped,sep=small]
    	{\K^{MW}_{\ast}(R_p)} && {\K^{W}_{\ast}(R_p)} \\
    	\\
    	{\K^{M}_{\ast}(R_p)} && {\K^{M}_{\ast}(R_p)/2} \\
    	& {\K^{MW}_{\ast}(Frac(R))} && {\K^{W}_{\ast}(Frac(R))} \\
    	\\
    	& {\K^{M}_{\ast}(Frac(R))} && {\K^{M}_{\ast}(Frac(R))/2} \\
    	&& {\K^{MW}_{\ast-1}(Frac(R/p))} && {\K^{W}_{\ast-1}(Frac(R/p))} \\
    	\\
    	&& {\K^{M}_{\ast-1}(Frac(R/p))} && {\K^{M}_{\ast-1}(Frac(R/p))/2}
    	\arrow[from=1-1, to=1-3]
    	\arrow[from=1-1, to=3-1]
    	\arrow[dotted, hook, from=1-1, to=4-2]
    	\arrow[from=1-3, to=3-3]
    	\arrow[dotted, hook, from=1-3, to=4-4]
    	\arrow[from=3-1, to=3-3]
    	\arrow[dotted, hook, from=3-1, to=6-2]
    	\arrow[dotted, hook, from=3-3, to=6-4]
    	\arrow[from=4-2, to=4-4]
    	\arrow[from=4-2, to=6-2]
    	\arrow["\partial"{description}, squiggly, from=4-2, to=7-3]
    	\arrow[from=4-4, to=6-4]
    	\arrow["\partial"{description}, squiggly, from=4-4, to=7-5]
    	\arrow[from=6-2, to=6-4]
    	\arrow["\partial"{description}, squiggly, from=6-2, to=9-3]
    	\arrow["\partial"{description}, squiggly, from=6-4, to=9-5]
    	\arrow[from=7-3, to=7-5]
    	\arrow[from=7-3, to=9-3]
    	\arrow[from=7-5, to=9-5]
    	\arrow[from=9-3, to=9-5]
    \end{tikzcd}
    \end{equation*}
    \normalsize
    Now \(\bigcap_{p \in R^{(1)}} \K^{MW}_{\ast}(R_p)\) is the limit in \(\mathrm{Ab}\) 
    \begin{equation*}
        \lim_{p \in R^{(1)}} ( \K^{MW}_{\ast}(R_p) \hookrightarrow \K^{MW}_{\ast}(Frac(R)) ).
    \end{equation*}
    By commuting limits we obtain the square in \(\SH(k)^{\heartsuit}\).\\ To obtain a cartesian square in \(\SH(k)\), we consider the pullback \(P\) in \(\SH(k)\) of
    \begin{equation*}
        \rmH\bfK^{M} \to \rmH\bfK^{M}/2 \leftarrow \rmH\bfK^{W}.
    \end{equation*}
    We get a long exact sequence of homotopy sheaves
    \begin{gather*}
        \cdots \to 0\to \pi_0(P) \to \pi_0(\rmH\bfK^{M}) \oplus \pi_0(\rmH\bfK^{W}) \to \pi_0(\rmH\bfK^{M}/2)  \to \pi_{-1}(P) \to 0 \to \cdots
    \end{gather*}
    which vanishes outside the illustrated part. But this sequence is 
    \begin{gather*}
         0\to \pi_0(P) \to \rmH\bfK^{M} \oplus \rmH\bfK^{W} \to \rmH\bfK^{M}/2  \to \pi_{-1}(P) \to 0 
    \end{gather*}
    and because both \(\mathrm{coker}(\rmH\bfK^{M} \to \rmH\bfK^{M}/2) \simeq 0\) and \(\mathrm{coker}(\rmH\bfK^{W} \to \rmH\bfK^{M}/2) \simeq 0\), \(\pi_{-1}(P)\) must vanish. Therefore, the map \(P \to \rmH\bfK^{MW}\) is an equivalence.
\end{proof}
\begin{remark}\label{rmk:fields}
    Let \(k\) be a perfect field.
    Using the same techniques as in the proof, we can show that every a map of unramified sheaves is an equivalence if it is an equivalence when valued on every finitely generated field extension of \(k\).
    In particular, this holds for all homotopy modules, and so for the homotopy sheaves \(\pi_i(-)_j\).
    It follows that a map \(f:E \to F \in \SH(k)\) is an equivalence if and only if \(\pi_i(f)_k\) is an equivalence when valued on fields. 
\end{remark}

\newpage
\section{Characteristic not 2, after Bachmann}\label{sec:BachmannProof}
Theorem \ref{thm:Morel-str-conj} in characteristic \(\neq 2\) is a reformulation of the following theorem.
\begin{theorem}[\protect{\cite[Thm. 17]{Bachmann_GenSlicesKO}}]\label{thm:MSC-char-not-two}
    Let \(k\) be a perfect field of characteristic \(p \neq 2\).
    The natural maps 
        \begin{equation*}
            \HZtilde \to \rmH\K^{MW} \quad \text{and} \quad \HWZ \to \rmH\K^{W}
        \end{equation*}
    induce isomorphisms on \(\pi_0(\bullet)_{\ast}\). 
    Moreover, the natural maps 
        \begin{equation*}
            \HZtilde \to \HZ \quad \text{and} \quad \HWZ \to \HZ/2 
        \end{equation*}
    induce isomorphisms on \(\pi_i(\bullet)_{\ast}\) for \(i \neq 0\).
\end{theorem}

\begin{proof}[\textbf{Proof of Theorem \ref{thm:Morel-str-conj} in characteristic \(p \neq 2\).}]\label{proof:MSC-char-not-two}
    Let \(k\) be a perfect field of characteristic \(p \neq 2\).
    Consider the pullback \(P\) of the span 
    \begin{equation*}
        \HZ \rightarrow \rmH \K^M \leftarrow \rmH \K^{MW};
    \end{equation*}
    by definition of \(\HZtilde\) there exists a canonical map \(\varphi: \HZtilde \to P\). 
    To prove that \(\varphi\) is an equivalence, it suffices to prove that \(\pi_i(\varphi)_\ast\) are isomorphisms for all \(i \in \bfZ\) since the homotopy t-structure is complete.
    There is a long exact sequence of sheaves of abelian groups
    \begin{equation*}
        \cdots \to \pi_{i+1}(\rmH\K^{M})_{\ast} \to \pi_{i}(P)_{\ast} \to \pi_{i}(\rmH\K^{MW})_{\ast} \oplus \pi_{i}(\HZ)_{\ast} \to \pi_{i}(\rmH\K^{M})_{\ast} \to \cdots.
    \end{equation*}
    \begin{itemize}
        \item For \(i < 0\), \(\pi_i(\rmH\K^{MW})_{\ast}\) and \(\pi_i(\rmH\K^{M})_{\ast}\) vanish because these two motivic spectra are in the heart of the t-structure, while \(\pi_i(\HZtilde)_{\ast}\) and \(\pi_i(\HZ)_{\ast}\) vanish because as in remark \ref{rmk:connectiveness-motivic-cohomology} motivic cohomologies are connective.
            It follows that \(\pi_i(P) \simeq 0\) for all \(i \leq -1\); thus, \(\pi_{i}(\varphi)_{\ast}\) are isomorphisms for all \(i \leq -1\).
        \item For \(i \geq 1\), \(\pi_i(\rmH\K^{MW})_{\ast}\) and \(\pi_i(\rmH\K^{M})_{\ast}\) still vanish because the two motivic spectra are in the heart of the t-structure. 
            So \(\pi_{i}(P) \to \pi_{i}(\HZ)\) are equivalences for all \(i \geq 1\). 
            Using the second part of Theorem \ref{thm:MSC-char-not-two}, \(\pi_i(\varphi)\) are equivalences for all \(i \geq 1\). 
        \item It remains an exact sequence
            \begin{equation*}
                0 \to \pi_0(P)_{\ast} \to \pi_{0}(\rmH\K^{MW})_{\ast} \oplus \pi_{0}(\HZ)_{\ast} \to \pi_0(\rmH\K^{M})_{\ast} \to \pi_{-1}(P)_{\ast} \to 0,
            \end{equation*}
            which is equivalent to
             \begin{equation*}
                 0 \to \pi_0(P)_{\ast} \longrightarrow \K^{MW}_{\ast} \oplus \; \K^{M}_{\ast} \overset{\chi - id}{\longrightarrow} \K^{M}_{\ast} \to 0,
             \end{equation*}
            where we are calling \(\chi\) the canonical map  \(\rmH\K^{MW} \to \rmH\K^{MW}/\eta \simeq\rmH\K^{M}\).
            Since the map \(\chi - id\) is surjective, \(\pi_{-1}(P)_{\ast}\) vanishes; thus \(\pi_{-1}(\varphi)_{\ast}\) is an isomorphism. 
            Finally, by verifying the universal property of the kernel, we can see that the induced map \(\pi_0(P)_{\ast} \to \pi_0(\HZ)_{\ast}\) must be an isomorphism:
            \begin{equation*}
                \begin{tikzcd}[cramped]
                	0 && {\K^{MW}_{\ast}} && {\K^{MW}_\ast \oplus \K^{M}_{\ast}} && {\K^{M}_{\ast}} && 0 \\
                	\\
                	0 && {\pi_0(P)_{\ast}} && {\K^{MW}_\ast \oplus \K^{M}_{\ast}} && {\K^{M}_{\ast}} && 0
                	\arrow[from=1-1, to=1-3]
                	\arrow["{(id,\chi)}", from=1-3, to=1-5]
                	\arrow[from=1-3, to=3-3]
                	\arrow["{\chi-id}", from=1-5, to=1-7]
                	\arrow["\simeq"{description}, from=1-5, to=3-5]
                	\arrow[from=1-7, to=1-9]
                	\arrow["\simeq"{description}, from=1-7, to=3-7]
                	\arrow[from=3-1, to=3-3]
                	\arrow[from=3-3, to=3-5]
                	\arrow["{\chi-id}", from=3-5, to=3-7]
                	\arrow[from=3-7, to=3-9]
                \end{tikzcd}
            \end{equation*}
            Using the first part of Theorem \ref{thm:MSC-char-not-two}, \(\pi_0(\varphi)_{\ast}\) is an equivalence.
    \end{itemize}
\end{proof}
\begin{proof}[\textbf{Proof of Theorem \ref{thm:Morel-str-conj-2} in characteristic \(p \neq 2\).}]\label{proof:MSC-2-char-not-two}
    The proof goes exactly the same as the proof of \ref{thm:Morel-str-conj} in characteristic \(p \neq 2\).
\end{proof}

\newpage
\section{Characteristic 2}\label{sec:Char2}

\subsection{The proof}
It remains to consider the characteristic \(2\) case.
We organize the study of \eqref{eq:str-conj} into two cases:
\begin{itemize}
    \item The square \eqref{eq:str-conj} is cartesian after reducing modulo \(2\).
    \item The square \eqref{eq:str-conj} is cartesian after inverting \(2\).
\end{itemize}
Both steps will be carried out under the assumption that \emph{Witt K-theory is effective in characteristic \(2\)}, which will be then proven in section \ref{subsec:Effectivity-WittK}.

For now, we assume the following key result:
\begin{theorem}\label{assumption}
    Let \(k\) be a perfect field of characteristic \(2\).
    The motivic spectrum \(\rmH \K^{W} \in \SH(k)\), representing Witt K-theory, is effective. 
\end{theorem}
\noindent We defer the proof of this to section \ref{subsec:Effectivity-WittK}.
\begin{remark}\label{rmk:smaller-assumptions}
    The assumption ``\(\rmH \K^W \) is effective'' holds if and only if the assumptions holds mod \(2\) and after inverting \(2\), i.e. if \(\rmH \K^W /2\) and \(\rmH \K^W[1/2]\) are effective.
    Effectivity is preserved under taking colimits, so the ``only if'' direction is clear.
    For the converse, notice that \(E \in \SH(k)\) is \(0\) if and only it is \(0\) mod \(2\), and after inverting \(2\).
    Indeed, the contractibility of the \(E/2\) implies the multiplications \(E \overset{2 \cdot }{\longrightarrow} E\) are equivalences; therefore, 
    \begin{equation*}
        E[1/2] = E \otimes \bfS[1/2] \overset{\simeq}{\longleftarrow} \colim_{n \in \bfN} (E \overset{\cdot 2}{\to} E\overset{\cdot 2}{\to} E \overset{\cdot 2}{\to} \dots ).
    \end{equation*}
    is equivalent to \(E\), as the former is the colimit of a direct system consisting solely of equivalences. 
    Applying this reasoning to the cofiber of the counit \(f^0(E) \to E\), and remembering \(\SH(k)\) is stable, shows the counit must be an equivalence, proving \(E\) is effective. 
\end{remark}

\begin{remark}
        Notice that effectivity after inverting \(2\) implies modulo \(\ell\) and tensor \(\bfQ\)\footnote{For \(E \in \SH(k)\), the spectra \(E \otimes \, \bfQ\) can be computed as the filtered colimit in \(\SH(k)\): 
        \begin{equation*}
            \colim_{n \in \bfN} (E \overset{\cdot 2}{\to} E \overset{\cdot 3}{\to} E \overset{\cdot 4}{\to} \dots ).
        \end{equation*}} 
        Indeed, assuming \(E \otimes\bfS[1/2] \simeq 0\), then
    \begin{equation*}
        E /\ell \overset{\simeq}{\longleftarrow} E \otimes \left(\bfS / \ell\right) \overset{\simeq}{\longleftarrow} E \otimes \left(\bfS[1/2] \otimes_{\bfS[1/2]} \bfS / \ell\right) \overset{\simeq}{\longleftarrow} \left(E \otimes\bfS[1/2]\right)  \otimes_{\bfS[1/2]}\bfS / \ell,
    \end{equation*}
    and
    \begin{equation*}
        E \otimes \bfQ \overset{\simeq}{\longleftarrow} E \otimes \left(\bfS[1/2] \otimes_{\bfS[1/2]} \bfQ \right) \overset{\simeq}{\longleftarrow} \left(E \otimes\bfS[1/2]\right)  \otimes_{\bfS[1/2]} \bfQ.
    \end{equation*} 
\end{remark}

\begin{lemma}\label{lemma:Morel-str-conj-modtwo}
    Let \(k\) be a perfect field of characteristic \(2\). Then 
    \begin{equation*}
        \begin{tikzcd}[cramped]
    	   {(\HZtilde)/2} && {\rmH \K^{MW}/2} \\
    	   \\
    	   {(\HZ)/2} && {\rmH \K^M/2}
    	   \arrow[from=1-1, to=1-3]
    	   \arrow[from=1-1, to=3-1]
    	   \arrow[from=1-3, to=3-3]
        	\arrow[from=3-1, to=3-3]
        \end{tikzcd}
    \end{equation*}
    is a cartesian square in \(\SH(k)\).
\end{lemma}
\begin{proof}
    Consider the square \eqref{eq:right-square} mod \(2\):
    \begin{equation*}
        \begin{tikzcd}
        	{\rmH\K^{MW}/2} &&& {\rmH\K^{W}/2} \\
        	\\
        	{\rmH\K^{M}/2} &&& {(\rmH\K^{M}/2)/2 \simeq \rmH\K^{M}/2 \oplus \rmH\K^{M}/2[1]}
        	\arrow[from=1-1, to=1-4]
        	\arrow[from=1-1, to=3-1]
        	\arrow[from=1-4, to=3-4]
        	\arrow[from=3-1, to=3-4]
        \end{tikzcd}
    \end{equation*}
    This square is still cartesian, as taking cofibers is exact in a stable category. The spectrum \(\rmH\K^{M}/2\) is effective; hence, its suspension \(\rmH\K^{M}/2 [1]\) and the coproduct \(\rmH\K^{M}/2 \; \oplus \; \rmH\K^{M}/2[1]\) are also effective.
    Theorem \ref{assumption} gives the effectivity \(\rmH \K^{W}/2\); we obtain that \(\rmH \K^{MW}/2\) must be effective as well, since it is the pullback of effective motivic spectra. Since \(\rmH \K^{MW}/2\) is effective, it is equivalent to its effective cover 
    \begin{equation*}
        \HZtilde/2 = f^0(\rmH \K^{MW}/2) \overset{\simeq}{\longrightarrow} \rmH \K^{MW}/2.
    \end{equation*}
    Consider now the pullback \(P_2\) of the span 
    \begin{equation*}
        \rmH\K^{MW}/2 \rightarrow \rmH\K^{M}/2 \leftarrow \HZ/2.
    \end{equation*}
    By definition of \(\HZtilde/2\) there exists a canonical map \(\varphi_{2}: \HZtilde/2 \to P_{2}\).
    Since the map \(\HZ/2 \to \rmH\K^{M}/2\) is an equivalence by \cite[Thm.8.4]{Geisser_Levine} (see Corollary \ref{thm:Geisser-Levine}), the map \(P_2 \to \rmH\K^{MW}/2\) also must be an equivalence, and so must be \(\varphi_2\). 
\end{proof}

\begin{lemma}\label{lemma:Morel-str-conj-modl}
    Let \(k\) be a perfect field of characteristic \(2\). The Morel structure conjecture holds in \(\SH(k)\) after inverting \(2\), i.e.
    \begin{equation}\label{eq:square-inv-two}
        \begin{tikzcd}[cramped]
	   {(\HZtilde)[1/2]} && {\rmH \K^{MW} [1/2]} \\
	   \\
	   {(\HZ)[1/2]} && {\rmH \K^M [1/2]}
	   \arrow[from=1-1, to=1-3]
	   \arrow[from=1-1, to=3-1]
	   \arrow[from=1-3, to=3-3]
    	\arrow[from=3-1, to=3-3]
    \end{tikzcd}
    \end{equation}
    is a cartesian square in \(\SH(k)\).
\end{lemma}

\begin{proof} 
    Consider the square \eqref{eq:right-square} tensored with \(\bfS[1/2]\).
    \begin{equation*}
        \begin{tikzcd}
        	{\rmH\K^{MW}[1/2]} &&& {\rmH\K^{W}[1/2]} \\
        	\\
        	{\rmH\K^{M}[1/2]} &&& {(\rmH\K^{M}/2)[1/2]}
        	\arrow[from=1-1, to=1-4]
        	\arrow[from=1-1, to=3-1]
        	\arrow[from=1-4, to=3-4]
        	\arrow[from=3-1, to=3-4]
        \end{tikzcd}
    \end{equation*}
    This square is still cartesian, as taking cofibers is exact in a stable category. Notice that \((\rmH\K^{M}/2) [1/2]\simeq 0\), so the square shows Milnor-Witt K-theory mod \(\ell\) as the direct sum 
    \begin{equation*}
        \rmH\bfK^{MW}[1/2]
        \overset{\simeq}{\longrightarrow} 
        \left(\rmH \bfK^{W} [1/2]\right) \oplus \left(\rmH \bfK^{M} [1/2] \right) 
        \overset{\simeq}{\longleftarrow}
        \left(\rmH \bfK^{W} \oplus \rmH \bfK^{M}\right) [1/2] 
    \end{equation*}
    We can now easily compute the pullback \(P_{1/2}\) of the span
    \begin{equation*}
        \rmH\K^{MW} \otimes \; \bfS[1/2]  \rightarrow \rmH\K^{M} \otimes \; \bfS[1/2] \leftarrow \HZ \otimes \; \bfS[1/2];
    \end{equation*}
    which is given by
    \begin{align*}
        &\rmH\K^{MW} [1/2] \oplus_{\rmH\K^{M} [1/2]} \HZ [1/2]
        \overset{\simeq}{\longrightarrow} \\
        &\overset{\simeq}{\longrightarrow}
        \left(\rmH\K^{M} [1/2] \oplus \rmH \K^{W} [1/2] \right) \oplus_{\rmH\K^{M} [1/2]} \HZ [1/2]   \overset{\simeq}{\longrightarrow} \\
        &\overset{\simeq}{\longrightarrow}
        \left(\rmH \K^{W} [1/2] \right) \oplus \left(\HZ [1/2]\right)
        \overset{\simeq}{\longleftarrow}
        \\
        &\overset{\simeq}{\longleftarrow}
        \left(\rmH \K^{W} \oplus \HZ \right) [1/2]
    \end{align*}
    By definition of \(\HZtilde[1/2]\) there exists a canonical map \(\varphi_{1/2}: \HZtilde[1/2] \to P_{1/2}\).
    Using the fact that the effective cover functor \(f^0\) is symmetric (and so it commutes with tensoring with \(\bfS[1/2]\)) and exact, we can also compute
    \begin{align*}
        \HZtilde[1/2] & = f^0(\rmH\K^{MW}) [1/2] \overset{\simeq}{\longrightarrow} f^0(\rmH\K^{MW}[1/2])\\
        &\overset{\simeq}{\longrightarrow} f^0(\rmH \K^{W}[1/2] \oplus \rmH\K^{M}[1/2] ) \overset{\simeq}{\longleftarrow}  f^0 ( \rmH \K^{W}[1/2] ) \oplus f^0( \rmH\K^{M}[1/2])\\
        &\overset{\simeq}{\longleftarrow} 
        f^0 ( \rmH \K^{W})[1/2] \oplus f^0( \rmH\K^{M})[1/2] \overset{\simeq}{\longleftarrow}  f^0 ( \rmH \K^{W})[1/2] \oplus \HZ [1/2].
    \end{align*}
    The map \(\varphi_{1/2}: \HZtilde[1/2] \to P_{1/2}\) is equivalent to the map \(f^0 ( \rmH \K^{W})[1/2] \oplus \HZ[1/2] \to \rmH \K^{W}[1/2] \oplus \HZ[1/2]\), which is the product of the counit \(f^0 \to id\) tensored with \(\bfS[1/2]\) and the identity map on \(\HZ[1/2]\).
    By Theorem \ref{assumption}, this counit map is an equivalence, so \(\varphi_{1/2}\) is also an equivalence and the square \eqref{eq:square-inv-two} is cartesian. 
\end{proof}

Now that we have proved the theorem modulo \(2\) and after inverting \(2\), we can finally prove Theorem \ref{thm:Morel-str-conj} (in characteristic \(2\)).
\begin{proof}[Proof of Theorem \ref{thm:Morel-str-conj} in characteristic \(2\).]
    Let \(k\) be a perfect field of characteristic \(2\).
    Consider the pullback \(P\) of the span 
    \begin{equation*}
        \HZ \rightarrow \rmH \K^M \leftarrow \rmH \K^{M};
    \end{equation*}
    by definition of \(\HZtilde\) there exists a canonical map \(\varphi: \HZtilde \to P\).
    \begin{equation*}
        \begin{tikzcd}[cramped,sep=scriptsize]
        	\HZtilde &&& P &&& {\rmH\K^{MW}} \\
        	\\
        	\\
        	&&& {\HZ} &&& {\rmH\K^{M}}
        	\arrow["{\varphi}"', dashed, from=1-1, to=1-4]
        	\arrow[curve={height=-18pt}, from=1-1, to=1-7]
        	\arrow[from=1-1, to=4-4]
        	\arrow[from=1-4, to=1-7]
        	\arrow[from=1-4, to=4-4]
        	\arrow[from=1-7, to=4-7]
        	\arrow[from=4-4, to=4-7]
        \end{tikzcd}
    \end{equation*}
    Since \(\varphi/2 \simeq \varphi_2\) and \(\varphi[1/2] \simeq \varphi_{1/2}\), and we have shown that these are equivalences, \(\varphi\) must be an equivalence as well by Remark \ref{rmk:smaller-assumptions}. 
    Therefore, square \eqref{eq:str-conj} is cartesian.
\end{proof}

In characteristic \(2\), Theorem \ref{thm:Morel-str-conj-2} follows immediately from Theorem \ref{assumption}.
\begin{proof}[Proof of Theorem \ref{thm:Morel-str-conj-2} in characteristic \(2\).]
      Let \(k\) be a perfect field of characteristic \(2\). 
      The maps \(\HZ/2 \to \rmH\bfK^{M}/2\) and \(\HWZ \to \rmH \bfK^{W}\) are both equivalence, the first by Theorem \ref{assumption}, the second by the Geisser-Levine Theorem (see Corollary \ref{cor:Geisser-Levine-motivic}).
      The square \eqref{eq:str-conj-2} is commutative, so it must be cartesian.
\end{proof}
\begin{remark}
    Theorem \ref{thm:Morel-str-conj-2} in characteristic \(2\) is equivalent to Theorem \ref{assumption}. 
    We have shown one direction.
    For the other direction, assume Theorem \ref{thm:Morel-str-conj-2} holds.
    By the Geisser-Levine Theorem, the map \(\HZ/2 \to \rmH\bfK^{M}/2\) is an equivalence. 
    Since the square is cartesian \(\HWZ \to \rmH \bfK^{W}\) is also an equivalence.
    So \(\bfK^{W}\) is effective, being equivalent to its own effective cover. 
\end{remark}

\subsection{Effectivity of Witt K-theory in char. 2}\label{subsec:Effectivity-WittK}
It remains for us to prove the Theorem \ref{assumption}.
To do it, we want to use the Bachmann-Fasel effectivity criterion \ref{thm:Bachmann-Fasel-effectivity}.

Fixed \(k\) perfect field of characteristic \(2\), to use the Bachmann-Fasel effectivity criterion \ref{thm:Bachmann-Fasel-effectivity} we need to show that 
\begin{equation*}
    \left\lvert \omega^{\infty}(\rmH\bfK^W \otimes \; \bfG_m^{\otimes d})( \hat{\Delta}^\bullet_F ) \right\rvert \simeq \ast,
\end{equation*}
for all \(d \geq 1\) and all finitely generated fields \(F/k\).
\begin{lemma}\label{lemma:hconiveau-HKW}
    Let \(k\) be a perfect field of characteristic \(2\). There are equivalences
    \begin{equation}\label{eq:fundamental-equiv}
        \omega^{\infty}(s_0 (\rmH \K^{W} \otimes \; \bfG_{m}^{\otimes d} ))(F) \overset{\simeq}{\longrightarrow} 
        \left\lvert \omega^{\infty}(\rmH \K^{W} \otimes \; \bfG_{m}^{\otimes d})(\hat{\Delta}^{\bullet}_F)\right\lvert 
        =
        \left\lvert R\Gamma_{Nis}(-, I^d)(\hat{\Delta}^{\bullet}_F)\right\rvert\
    \end{equation}
    for each field \(F\) finitely generated over \(k\), for each \(n \geqslant 0\).
\end{lemma}
\begin{proof}
    By the homotopy coniveau tower, see Theorem \ref{thm:h-coniveau}, we have an equivalence
    \begin{equation*}
        \omega^{\infty}(s_0 (\rmH \K^{W} \otimes \; \bfG_{m}^{\otimes d}))(F) \to \left\lvert \omega^{\infty}(\rmH \K^{W} \otimes \; \bfG_{m}^{\otimes d})(\hat{\Delta}^{\bullet}_F)\right\lvert 
    \end{equation*}
    The equality is by definition of the functor \(H: \bfH \bfI(k) \to \SH(k)\). 
\end{proof}
\begin{lemma}\label{lemma:s_0-KW}
    Let \(k\) be a perfect field of characteristic \(2\). For all \(d \geq 1\), the canonical map \(H\bfK^W \otimes \, \bfG_m \overset{\eta \cdot}{\longrightarrow} H\bfK^W\) induces an equivalence on \(0\)-slices 
    \begin{equation}
        s_0 (\rmH\bfK^W \otimes \, \bfG_m^{\otimes d+1}) \overset{\simeq}{\longrightarrow} s_0 (\rmH\bfK^W \otimes \, \bfG_m^{\otimes d}).
    \end{equation}
\end{lemma}
\begin{proof}
    Since tensor products are exact in a stable category, the sequence \eqref{eq:short-exact-seq-Kato} remains a cofiber sequence after tensoring with \(\bfG_m^{\otimes d}\).
    Taking \(0\)-slices is exact, so we get a cofiber sequence 
        \begin{equation*}
            s_0 (\rmH\bfK^W \otimes \, \bfG_m^{\otimes d+1} ) \longrightarrow s_0 (\rmH\bfK^W \otimes \, \bfG_m^{\otimes d}) \longrightarrow s_0 (\rmH\bfK^M/2 \otimes \, \bfG_m^{\otimes d}).
        \end{equation*}
    Since \(\rmH\bfK^M/2\) is effective, then, for all \(d \geq 1\),
        \begin{equation*}
            0 \overset{\simeq}{\longrightarrow} s_{-d} (\rmH\bfK^M/2) \otimes \, \bfG_m^{\otimes d} \overset{\simeq}{\longrightarrow} s_0 (\rmH\bfK^M/2 \otimes \, \bfG_m^{\otimes d}).
        \end{equation*}
    The third object of the first sequence is \(0\), since it is the same as the third object of the first one; consequently, the first map of the first sequence is an equivalence.
\end{proof}
\noindent Using Lemma \ref{lemma:hconiveau-HKW} and Lemma \ref{lemma:s_0-KW}, we obtain equivalences of spectra
\begin{equation*}
    \calX\coloneq\left\lvert R\Gamma_{Nis}(-, I)(\hat{\Delta}^\bullet_F) 
    \right\rvert \overset{\simeq}{\longleftarrow} 
    \left\lvert R\Gamma_{Nis}(-, I^2)(\hat{\Delta}^\bullet_F)
    \right\rvert \overset{\simeq}{\longleftarrow}
    \cdots
\end{equation*}
We want to show that \(\calX\) is contractible.
First, we need some information about \(H^i_{Nis}(\hat{\Delta}^m_F, I^d)\).

\begin{remark}[ \protect{\cite[Example 12.2]{Voevodsky_LectureNotes}}]
    Let \(F\) be a field. Then \(F\) has Nisnevich cohomological dimension zero. In particular, if \(F\) is a finitely generated field over a perfect field \(k\), and we consider the homotopy module \(\bfI^* \in \Shv_{Nis}(\Smk, \mathrm{Ab})^{\bfZ}\), then \(H^i_{Nis}(F, I^d)=0\) for all \(i>0\). More generally, this holds for all henselian local rings over \(k\). 
\end{remark}

\begin{lemma}[(Injectivity for semilocal rings)] \label{lemma:semilocality}
    Let \(k\) be a perfect field. For each field \(F\)  finitely generated over \(k\) and for each \(m \geq 0\), then
    \begin{equation*}
        I^d (\hat{\Delta}^m_F)= H^0_{Nis}(-, I^d)(\hat{\Delta}^m_F) \hookrightarrow H^0_{Nis}(F, I^d) = I^d (F(t_0, t_1,\dots, t_m))
    \end{equation*}
    and, for all \(i>0\)
    \begin{equation*}
        H^i_{Nis}(\hat{\Delta}^m_F, I^d) \hookrightarrow H^i_{Nis}(F(t_0, t_1,\dots, t_m), I^d)=0.
    \end{equation*}
\end{lemma}
\begin{proof}
    Let \(k\) be a perfect field and \(\calF: \mathrm{Reg}_k^{op} \to \mathrm{Sp}\) be a finitary, Nisnevich sheaf on regular Noetherian \(k\)-schemes such that whose restriction to smooth \(k\)-scheme is deflatable
    \footnote{Let \(k\) be any field, and a presheaf \(\calF: \Smk^{opp} \to \Spectra\).
    Consider the two morphism of presheaves \(j^*, \pi^* \infty^* \calF(\bfP^1 \times_k -) \to \calF(\bfA^1 \times_k -)\), where \(\pi: \bfA^1 \times_k X \to X\) is projection, \(\infty: \Spec(k) \to \bfP^1\) is the inclusion of the point at infinity, and \(j: \bfA^1 \hookrightarrow \bfP^1\) is the open immersion complementary to the point at \(\infty\). \(F\) is called deflatable if \(j^*\) and \(\pi^* \infty^*\) are homotopic.}
    Then for any \(n \in \bfZ\) and any regular local \(k\)-algebra \(R\).
    By \cite[Lemma 3.11]{elmanto2026motiviccohomologyequicharacteristicschemes}, the canonical map
    \begin{equation*}
        \pi_n(\calF(R)) \to \pi_n(\calF(Frac(R)))
    \end{equation*}
    is injective.
    We can extend this to essentially smooth schemes because homotopy groups of spectra commute with filtered colimits, and filtered colimits are exact.
    Alternatively, we know there is a short exact sequence of homotopy modules
    \begin{equation*}
        0 \to \bfI^{\ast} \to \bfK^{MW}_{\ast} \to \bfK^{M}_{\ast} \to 0,
    \end{equation*}
    i.e. a fiber sequence in \(\SH(k)\)
    \begin{equation*}
        \rmH \bfI \to \rmH \bfK^{MW} \to \rmH \bfK^{M}.
    \end{equation*}
    Let \(X\) be a semilocal scheme, for every \(d \geq 0\) we get diagrams of long exact sequences
    \begin{equation*}
        \begin{tikzcd}[cramped,column sep=tiny,row sep=scriptsize]
        	{0 } && {\bfI^{d+1}(X)} && {\bfK^{MW}_d(X)} && {\bfK^{M}_d(X)} && {H^1_{Nis}(X, \bfI^{d+1})} && {H^1_{Nis}(X, \bfK^{MW}_{d})} && \dots \\
        	\\
        	0 && {\bfI^{d+1}(\kappa(X))} && {\bfK^{MW}_d(\kappa(X))} && {\bfK^{M}_d(\kappa(X))} && {H^1_{Nis}(\kappa(X), \bfI^{d+1})} && {H^1_{Nis}(\kappa(X), \bfK^{MW}_{d})} && \dots
        	\arrow[from=1-1, to=1-3]
        	\arrow[from=1-3, to=1-5]
        	\arrow[from=1-3, to=3-3]
        	\arrow[from=1-5, to=1-7]
        	\arrow[from=1-5, to=3-5]
        	\arrow[from=1-7, to=1-9]
        	\arrow[from=1-7, to=3-7]
        	\arrow[from=1-9, to=1-11]
        	\arrow[from=1-9, to=3-9]
        	\arrow[from=1-11, to=1-13]
        	\arrow[from=1-11, to=3-11]
        	\arrow[from=3-1, to=3-3]
        	\arrow[from=3-3, to=3-5]
        	\arrow[from=3-5, to=3-7]
        	\arrow[from=3-7, to=3-9]
        	\arrow[from=3-9, to=3-11]
        	\arrow[from=3-11, to=3-13]
        \end{tikzcd}
    \end{equation*}
    Each of the vertical maps involving \(\bfK^{MW}_d\) is injective by \cite[Ch. 7. Cor. 3.3]{Bachmann_MWMotives}, and each of the vertical maps involving \(\bfK^{M}_d\) is injective by \cite[7.3.5]{ColliotThelene-BlochOgusGabber}.
    The sequence of kernels must vanish, so each vertical map involving \(\bfI^{d+1}\) is injective. 
\end{proof}

The mechanism by which we will prove that \(\calX\) is contractible is explained in the following remark.
\begin{remark}\label{rmk:connectiveness-simplicial-spectrum}
    Let \(K_\bullet\) be simplicial spectrum  and \(j \geq 0\). 
    Suppose that, for every \(m \geq 0\), the spectrum \(K_m\) is \((j - m)\)-connective
    \footnote{Meaning that the homotopy groups \(\pi_i(K_m)\) vanish for \(i < j-m\).}.
    Then the geometric realization \(|K_\bullet|\) is \(j\)-connective. This follows from \cite{Lurie_DAGVIII}[Remark 4.3.4].
\end{remark}
\noindent Therefore, to prove that \(\calX\) is \(j\)-connective for all \(j\), it is enough to show that \(R\Gamma_{Nis}(-, I^d)(\hat{\Delta}^m_F)\) is \((j-m)\)-connective, for some \(d\) that we do not have a priori to specify. We can prove there exists a \(d\) such that \(R\Gamma_{Nis}(-, I^d)(\hat{\Delta}^m_F)\) is contractible.
To this end, the following two results will be useful.

\begin{theorem}[\protect{\cite[Thm. 5]{Milnor_SimmetricInnerProducts}}]\label{thm:Milnor-Imperfection-theorem}
    Let \(F\) be a field of characteristic \(2\).
    If the degree\footnote{The dimension of \(F\) as a \(F^2\)-vector space.} of \(F\) over its subfield of squares \(F^2\) is \(d=2^j\), then the ideal \(I^j(F)\) is non-zero, but \(I^{j+1}(F) = 0\).
\end{theorem} 
The integer \(j\) is called \emph{degree of imperfection} of \(F\).
\begin{remark}
    If \(F\) is finitely generated over a perfect field \(k\) of characteristic \(2\) (or any field of characteristic \(p\)), the degree of imperfection of \(F\) corresponds to the transcendence degree of \(F\) over \(k\).
    In particular, we see that in this case, the degree of imperfection is invariant under finite extensions and increases by \(1\) under simple transcendental extensions.
    
    To prove this, consider the exact sequence of K\"ahler modules
    \begin{equation*}
        \Omega^1_{k/\bfZ} \otimes_k F \to \Omega^1_{F/\bfZ} \to \Omega^1_{F/k} \to 0.
    \end{equation*}
    Since \(k\) is perfect, \(\Omega^1_{k/\bfZ}\) must vanish. 
    We also know that the rank of the K\"ahler module \(\Omega^{1}_{F/k}\) corresponds to the transcendental degree of \(F/k\).
    Finally, we need to show that the rank of the K\"ahler module \(\Omega^{1}_{F/\bfZ}\) corresponds to the degree of imperfection of \(F\).
    Notice that \(F/F^2\) is purely inseparable finite
    extension. 
    Therefore, we can find a elements \(x_1, \dots, x_j\) such that \(\{x_1^{i_1} x_2^{i_2} \cdots x_j^{i_j} \mid i_1, i_2, \dots, i_j \in \{0,1\}\}\) is an \(F^2\)-basis of \(F\). 
    We claim \(\{dx_1, \dots, dx_j\}\) is an \(F\)-basis for \(\Omega^1_{F/F^2}\). 
    Using Leibniz's rule, we can immediately say these elements are generators. 
    To prove they are linearly independent, let's start by constructing a derivation \(\Pi^i\) by formally differentiating the expression \(\sum a_{i_1 \cdots i_n} x_1^{i_1} x_2^{i_2} \cdots x_j^{i_j}\) with respect to \(x_i\). 
    This is an \(F^2\)-derivation, so it induces a factorization of this map as \(F \overset{d_{F/\bfZ}}{\longrightarrow} \Omega^1_{F/\bfZ} \overset{\pi^i}{\longrightarrow} F\), where \(\pi^i(dx_i)=1\) and \(\pi^i(dx_j)=0\) for \(j \neq i\). 
    Next, assume there is a linear combination \(\sum a_i d x_i=0\) with \(a_i \in F\). 
    So, \(0=\pi^i(0)=\pi^i(\sum a_i d x_i)=ai\), from which we can conclude all \(a_i\)'s are zero.
    So we have proved \(j=\dim_F(\Omega^1_{F/F^2})\) and
    there is an isomorphism of \(F\)-vector spaces \(\Omega^1_{F/F^2} \to \Omega^1_{F/\bfZ}\) given by \(dx \mapsto dx\) (as all \(2\)-powers have already zero differential).
\end{remark}

\begin{corollary}\label{corollary:vanishing-fundamental-ideal}
     Let \(k\) be a perfect field of characteristic \(2\), \(F\) be a finitely generated field over \(k\), and \(m\) a non-negative integer. For all \(d > m + tr.deg.(F/k)+ 1\)
     \begin{equation*}
         R\Gamma_{Nis}(-, I^d)(\hat{\Delta}^m_F) \simeq 0.
     \end{equation*}
\end{corollary}
\begin{proof}
    For each \(n > 0\), \(\pi_n(R\Gamma_{Nis}(-, I^d)(\hat{\Delta}^m_F))\) vanish, since the spectrum is coconnective (there is no negative cohomology).
    For each \(n < 0 \), \(\pi_n(R\Gamma_{Nis}(-, I^d)(\hat{\Delta}^m_F)) \simeq H^n_{Nis}(-, I^d)(\hat{\Delta}^m_F)\) vanish by Lemma \ref{lemma:semilocality}.
    Finally,  by Lemma \ref{lemma:semilocality}, \(\pi_0(R\Gamma_{Nis}(-, I^d)(\hat{\Delta}^m_F))\) injects into \(I^d(F(t_0, \dots, t_m))\) which vanish by Lemma \ref{thm:Milnor-Imperfection-theorem} since  \(d> m + tr.deg.(F/k)+1\). 
    Therefore, \(R\Gamma_{Nis}(-, I^d)(\hat{\Delta}^m_F)\) is contractible.
\end{proof}

\begin{proof}[\textbf{Proof of Theorem }\ref{assumption}]
    Fix a perfect field \(k\) of characteristic \(2\). We want to use the Bachmann-Fasel effectivity criterion \ref{thm:Bachmann-Fasel-effectivity} to prove \(\rmH\bfK^W\) is effective. We need to verify
    \begin{equation*}
        \left\lvert (\rmH\bfK^W \otimes \; \bfG_m^{\otimes d})( \hat{\Delta}^\bullet_F ) \right\rvert = \left\lvert R\Gamma_{Nis}(-, I^d) (\hat{\Delta}^{\bullet}_F)\right\rvert \simeq \ast,
    \end{equation*}
    for all \(d \geq 1\) and all finitely generated fields \(F/k\).
    Using the homotopy coniveau tower (see Lemma \ref{lemma:hconiveau-HKW}), we have the following equivalence.
    \begin{equation*}
        s_0 (\rmH \K^{W} \otimes \; \bfG_{m}^{\otimes d})(F)
        \overset{\simeq}{\longrightarrow}  
        \left\lvert R\Gamma_{Nis}(-, I^d (\hat{\Delta}^{\bullet}_F)\right\rvert;
        \end{equation*}
    together with Lemma \ref{lemma:s_0-KW},
    \begin{equation*}
        s_0 (\rmH\bfK^W \otimes \, \bfG_m^{\otimes d+1})  
        \overset{\simeq}{\longrightarrow} 
        s_0 (\rmH\bfK^W \otimes \, \bfG_m^{\otimes d})
    \end{equation*}
    we obtain equivalences of spectra
    \begin{equation*}
        \calX\coloneq\left\lvert R\Gamma_{Nis}(-, I)(\hat{\Delta}^\bullet_F) 
        \right\rvert \overset{\simeq}{\longleftarrow} 
        \left\lvert R\Gamma_{Nis}(-, I^2)(\hat{\Delta}^\bullet_F)
        \right\rvert \overset{\simeq}{\longleftarrow}
        \cdots
    \end{equation*}
    Because of Remark \ref{rmk:connectiveness-simplicial-spectrum}, proving  \(\calX\) is contractible is equivalent to proving that for a fixed \(m \geq 0\) there exists \(d \geq 1\) such that \(R\Gamma_{Nis}(-, I^d)(\hat{\Delta}^m_F) \simeq 0\).
    Corollary \ref{corollary:vanishing-fundamental-ideal} proves such a \(d\) can always be found. Therefore, \(\rmH \K^{W}\) is effective.
\end{proof}
 
\newpage
\printbibliography
\end{document}